\documentclass{amsart}
\usepackage[utf8x]{inputenc}
\usepackage{amsmath, amsthm, amssymb}
\usepackage[usenames,dvipsnames,svgnames,table]{xcolor}
\usepackage[margin=1.18in]{geometry}
\usepackage{enumerate}
\usepackage{dsfont}
\usepackage{cleveref}
\usepackage{color}

\newtheorem{theorem}{Theorem}[section]
\newtheorem{proposition}[theorem]{Proposition}
\newtheorem{lemma}[theorem]{Lemma}

\newtheorem{corollary}[theorem]{Corollary}

\newcommand{\N}{\mathbb N}

\newcommand{\R}{\mathbb R}
\newcommand{\bbS}{{\mathbb S}^2}
\newcommand{\T}{\mathbb T}
\newcommand{\Z}{\mathbb Z}

\newcommand{\cA}{\mathcal{A}}

\newcommand{\cL}{\mathcal{L}}

\newcommand{\cT}{\mathcal{T}}

\newcommand{\eps}{\varepsilon}
\newcommand{\e}{\eps}
\newcommand{\de}{\delta}

\newcommand{\dd}{\, \mathrm{d}}

\newcommand{\1}{\mathds{1}}
\newcommand{\vv}{\langle v\rangle}
\newcommand{\ww}{\langle w \rangle}
\newcommand{\vvp}{\langle v' \rangle}

\newcommand{\ddt}{\frac{\dd}{\dd t}}

\newcommand{\Qs}{Q_{\rm s}}
\newcommand{\Qns}{Q_{\rm ns}}
\newcommand{\dom}{\T^3\times\R^3}

\DeclareMathOperator{\pv}{p.v.}

\numberwithin{equation}{section}

\crefname{equation}{}{}

\title[Solutions of the Boltzmann equation with slow decay]{Local well-posedness of the Boltzmann equation with polynomially decaying initial data}
\author{Christopher Henderson}
\email{ckhenderson@math.arizona.edu}
\thanks{CH: Department of Mathematics, University of Arizona, Tucson, AZ 85721}
\author{Stanley Snelson}
\email{ssnelson@fit.edu}
\thanks{SS: Department of Mathematical Sciences, Florida Institute of Technology, Melbourne, FL 32901}
\author{Andrei Tarfulea}
\email{tarfulea@lsu.edu}
\thanks{AT: Department of Mathematics, Louisiana State University, Baton Rouge, LA 70803}

\thanks{CH was partially supported by NSF grant DMS-1907853. SS was partially supported by a Ralph E. Powe Award from ORAU. AT was partially supported by NSF grant DMS-1816643}

\begin{document}

\begin{abstract}
We consider the Cauchy problem for the spatially inhomogeneous non-cutoff Boltzmann equation with polynomially decaying initial data in the velocity variable.  We establish short-time existence for any initial data with this decay in a fifth order Sobolev space by working in a mixed $L^2$ and $L^\infty$ space that allows to compensate for potential moment generation and obtaining new estimates on the collision operator that are well-adapted to this space.  Our results improve the range of parameters for which the Boltzmann equation is well-posed in this decay regime, 
as well as relax the restrictions on the initial regularity. As an application, we can combine our existence result with the recent conditional regularity estimates of Imbert-Silvestre (arXiv:1909.12729 [math.AP]) to conclude solutions can be continued for as long as the mass, energy, and entropy densities remain under control. This continuation criterion was previously only available in the restricted range of parameters of previous well-posedness results for polynomially decaying initial data.
\end{abstract}

\maketitle

\section{Introduction}

We are concerned with the Boltzmann equation, a kinetic integro-differential equation that models the particle density $f(t,x,v)$ of a diffuse gas in phase space. A solution $f:[0,T]\times \T^3 \times \R^3\to \R_+$ satisfies
\begin{equation}\label{e:boltzmann}
	(\partial_t + v\cdot\nabla_x) f = Q(f,f), 
\end{equation}
where the bilinear collision operator is defined for functions $f_1,f_2:\R^3\to \R_+$ by
\begin{equation}\label{e:Q}
Q(f_1,f_2) = \int_{\R^3}\int_{\mathbb S^{2}} (f_1(v_*')f_2(v') - f_1(v_*)f_2(v))B(|v-v_*|, \cos\theta)  \dd \sigma \dd v_*,
\end{equation}
with 
\begin{equation}\label{e:Q'}
v' = \frac{v+v_*}{2}+\frac{|v-v_*|}{2}\sigma, \qquad v_*' = \frac{v+v_*}{2}-\frac{|v-v_*|}{2}\sigma, \quad \text{ and }
\cos\theta  = \frac{v-v_*}{|v-v_*|}\cdot \sigma.
\end{equation}
Here, $\theta$ represents the angle between the post-collisional velocities $v$ and $v_*$, and the pre-collisional velocities $v'$ and $v_*'$.

The collision kernel $B(|v-v_*|, \cos\theta)$ depends on the interaction potential between particles. The common choice of an inverse power law potential $\phi(r) = 1/r^{\alpha-1}$ for some $\alpha>2$ leads to a kernel of the form
\begin{equation}\label{e:B}
B(|v-v_*|,\cos\theta) = |v-v_*|^\gamma b(\cos\theta), \qquad b(\cos\theta)\approx |\theta|^{-2-2s},
\end{equation}
where $\gamma = (\alpha-5)/(\alpha-1)$ and $s = 1/(\alpha-1)$. We disregard the parameter $\alpha$ and consider arbitrary pairs $\gamma> -3$, $s\in (0,1)$, which is fairly common in the mathematical literature on the Boltzmann equation. More specifically, we restrict our attention to the case
\[ \max\left\{-3,-\frac 3 2 - 2s\right\} < \gamma < 0, \]
with $s\in (0,1)$ arbitrary. See Section \ref{s:ideas} below for an explanation of the technical reasons for the restriction on $\gamma$.

The angular cross-section $b(\cos\theta)$ has a non-integrable singularity for \emph{grazing collisions} (i.e. collisions with $\theta \approx 0$). This singularity reflects the fact that long-range interactions are taken into account by the collision operator. The version of \eqref{e:boltzmann} that includes this singularity is called the \emph{non-cutoff} Boltzmann equation, and that is the case of interest here.

In this article, we address the local well-posedness theory for \eqref{e:boltzmann}. The key feature of our study is that the initial data is allowed to have merely polynomial (rather than exponential or Gaussian) decay in the velocity variable. The only comparable prior result we are aware of is the work of Morimoto-Yang \cite{morimoto2015polynomial}, which proved existence under similar assumptions, but only addressed the regime $s \in (0, 1/2)$ and $\gamma \in (- 3/2, 0]$. 
We feel it is important to fill this gap in the literature, because $s\in [1/2,1)$ is the most delicate case, at least with respect to the severity of the grazing collisions singularity. This can readily be seen from \eqref{e:B}.  We should mention that uniqueness in a polynomially-weighted Sobolev space, but not existence, was proven in \cite{amuxy2011uniqueness} for the same range of $\gamma$ and $s$ that we consider here. 

Compared to \cite{morimoto2015polynomial}, we also improve the required number of derivatives of $f_{\rm in}$ from six to five. By avoiding any cutoff-based approximation, we provide a method for constructing solutions in weighted Sobolev spaces that is different from, and arguably simpler than, the methods commonly seen in the literature.


To state our results precisely, we need to define the following function spaces. With $\vv = \sqrt{1+|v|^2}$, define the weighted $L^p$ and Sobolev norms
\[
	\|g\|_{L^{p,n}} = \|\vv^n g\|_{L^p}
		\quad \text{ and }\quad
	\|g\|_{H^{k,n}} = \|\vv^n g\|_{H^k}.
\]
Next, for $k,n,m \geq 0$, define
\[
	X^{k,n,m} := H^{k,n} \cap L^{\infty,m}(\R^6)
		\quad\text{ and } \quad
	Y^{k,n,m}_T :=L^\infty_t([0,T] ; X^{k,n,m}).
\]

Notice that if $f\in Y^{k,n,m}_T$ solves~\eqref{e:boltzmann}, then $f \in W^{1,\infty}([0,T]; H^{k-2s,n-(\gamma + 2s)_+})$ as well (see, e.g., the estimates in \cite[Theorem 2.1]{amuxy2010regularizing}).   For succinctness, we sometimes denote $L^\infty([0,T]; H^{k,n})$ as $L^\infty_t H^{k,n}_{x,v}$ (similarly for other well-known Banach spaces).

Our main result is

\begin{theorem}\label{t:main}
	If $0 \leq f_{\rm in} \in X^{k,n,m}$ with $k\geq 5$, $n> 3/2 + (\gamma+2s)_+$, and $m$ sufficiently large depending on $k$, $n$, $\gamma$, and $s$, then there exists a time $T > 0$ (depending only on $\| f_{\rm in} \|_{X^{k,n,m}}$, $\gamma$, $s$, $k$, $n$, and $m$) and a unique $f\in Y_T^{k,n,m} \cap C([0,T]; H_{x,v}^{k,n})$ such that $f$ solves~\eqref{e:boltzmann} and $f(0,\cdot,\cdot) = f_{\rm in}$.  Moreover, $f\geq 0$.
\end{theorem}

As mentioned above, the uniqueness part of this theorem was established in \cite{amuxy2011uniqueness}.

Note that the assumption that $f_{\rm in} \in L^{\infty,m}$ is not a requirement on the regularity of $f$ since $H^{k,n}$ embeds into $L^{\infty,n}$; instead, it is only a requirement on the decay of $f$ for large velocities.

One application of Theorem \ref{t:main} is as follows: by the conditional regularity theorem recently established in \cite{imbert2019smooth}, when $s \in (0,1/2)$ and $\gamma \in (-3/2,0]$, spatially periodic solutions can be extended past a given time $T$ provided the mass, energy, and entropy densities of $f$ are uniformly bounded above, and the mass density is uniformly bounded below; in other words,
\begin{equation}\label{e:hydro}
 0< m_0 \leq \int_{\R^3} f(t,x,v)\dd v \leq M_0, \  \int_{\R^3} |v|^2 f(t,x,v) \leq E_0, \  \mbox{and } \int_{\R^3}f(t,x,v) \log f(t,x,v) \dd v \leq H_0,
 \end{equation}
uniformly in $t\in [0,T], x\in \R^3$.  The main idea is to pair the conditional regularity estimates~\cite{imbert2018schauder, imbert2018decay} with a well-adapted short-time well-posedness result.  In~\cite{imbert2019smooth}, the authors use the well-posedness result of~\cite{morimoto2015polynomial}, which is the cause of their restriction on $s$ and $\gamma$.  In view of \Cref{t:main}, this can now be applied to $s \in (0,1)$ and $\gamma \in (\max\{-3/2 - 2s, -3\}, 0)$ as well.  That is:

\begin{corollary}\label{cor}
With $\gamma$ and $s$ as above, assume in addition that $\gamma +2s \in [0,2]$. Let $f$ be any solution $f$ of \eqref{e:boltzmann} with initial data in Schwartz class, i.e. $\vv^\ell \partial_x^\alpha \partial_v^\beta f_{\rm in}\in L^\infty(\mathbb T^3, \R^3)$ for all $\ell \geq 0$ and any multi-indices $\alpha,\beta \in \mathbb N^3$. Then $f$ can be extended for as long as the condition \eqref{e:hydro} holds. In other words, if $f$ exists on $[0,T)\times \mathbb T^3 \times \R^3$ but cannot be extended to a solution on $[0,T+\eps)\times \mathbb T^3\times \R^3$ for any $\eps>0$, then at least one of the inequalities in \eqref{e:hydro} must degenerate as $t \to T$.
\end{corollary}

The fact that our initial data is allowed to decay merely polynomially is crucial in establishing \Cref{cor} because polynomial decay in $v$ in the initial data is propagated forward in time (see \cite{imbert2018decay}) but it is not currently known if the same is true for Gaussian decay. 
Since the proof of \Cref{cor} is exactly as in the previously established case of~\cite{imbert2019smooth}, we omit it. At the cost of more technical arguments, we expect this continuation criterion can be extended to any solution with initial data in $X^{k,n,m}$ with $m$ sufficiently large.

It would be interesting to extend our existence result to the two remaining cases, $\gamma \in (-3,\max(-3/2-2s,-3)]$ and $\gamma \geq 0$; we leave this for future work. Another open issue is decreasing the required regularity of the initial data, as in our recent work \cite{HST2019rough} on the closely related Landau equation, which required only polynomial decay in $v$ and H\"older continuity for $f_{\rm in}$ to establish existence and uniqueness (and if the assumption of H\"older continuity is dropped, we can establish existence but not uniqueness).

\subsection{Related work}\label{s:related}

We focus mainly on the non-cutoff Boltzmann equation in this discussion. The cutoff case has its own long history that we omit here (see, for example \cite[Chapter 2, Section 3]{villani2002review} and the references therein).

Existence results for the non-cutoff Boltzmann equation have come in the following flavors:

\begin{itemize}
	\item {\bf Spatially homogeneous solutions.} Many more results on existence and regularity are available in the case when $f$ and $f_{\rm in}$ are independent of $x$: see  \cite{chen2011boltzmann, DM2005Lp, desvillettes2004smooth, goudon, villani1998homogeneous} and the references therein. Even in this case, global-in-time existence has only been proven when $\gamma \geq -2$: see \cite{desvillettes2009homogeneous}. We should also mention measure-valued solutions, which are known to exist globally and regularize depending on the value of $\gamma$~\cite{lu2012measure,morimoto2016measure}.
	
	\item {\bf Weak solutions.} For the full, inhomogeneous equation, renormalized solutions with defect measure were constructed by Alexandre-Villani \cite{alexandre2002longrange}, see also \cite{alexandre2004landau}. These solutions are a generalization of the renormalized solutions first constructed by DiPerna-Lions \cite{diperna1989cauchy} for the cutoff Boltzmann equation. The uniqueness and regularity of these solutions are not understood.

	\item {\bf Close-to-equilibrium solutions.} When  $f_{\rm in}$ is sufficiently close to a Maxwellian equilibrium state ($c_1 e^{-c_2|v|^2}$ with $c_1,c_2>0$) in an appropriate norm, solutions  exist globally and converge to equilibrium as $t\to \infty$: see for example \cite{amuxy2011hardpotentials,alexandre2012global,amuxy2011global,  alonso2018non, gressman2011boltzmann,he2017global,herau2019regularization}. This is the only setting in which global, classical solutions to the inhomogeneous equation are known to exist.
	
	
	\item {\bf Short-time solutions.} There are several existence results for solutions on a time interval $[0,T]$, for example \cite{alexandre2001solutions, amuxy2010regularizing, amuxy2011bounded, amuxy2013mild}. Most commonly, the initial data is required to have Gaussian decay in velocity, and lie in a Sobolev space of order 4 or higher. In \cite{morimoto2015polynomial}, the authors weaken the decay assumption from Gaussian to polynomial at the expense of working in a weighted $H^6$ space.
\end{itemize}

Global existence of spatially inhomogeneous solutions with non-perturbative initial data is a very challenging open problem. In the last few years, good progress has been made on \emph{conditional regularity} for solutions of non-cutoff Boltzmann under the physically relevant assumptions \eqref{e:hydro}. Silvestre \cite{silvestre2016boltzmann} established \emph{a priori} $L^\infty$ bounds, Imbert-Silvestre \cite{imbert2016weak} established $C^\alpha$ regularity, Imbert-Mouhot-Silvestre \cite{imbert2018decay} established polynomial decay estimates as $|v|\to \infty$, and Imbert-Silvestre \cite{imbert2019smooth} finally showed $C^\infty$ regularity. As mentioned above, our Theorem \ref{t:main} combined with the result of \cite{imbert2019smooth} establishes \eqref{e:hydro} as a continuation criterion for solutions.

\subsection{Proof ideas}\label{s:ideas}

The key tool in our proof is an \emph{a priori} estimate, given a fixed $g\in Y_T^{k,n,m}$, in the $Y_T^{k,n,m}$ norm for solutions of the linear Boltzmann equation
\[
	\partial_t f + v\cdot \nabla_x f = Q(g,f).
\]
This estimate comes from the combination of polynomially-weighted $L^2$ estimates (that are proven by the energy method) and polynomially-weighted $L^\infty$ estimates (that are proven using comparison principle arguments). Interpolating between these two estimates compensates for the moment loss generated by the collision operator $Q$ in the energy estimates. A key observation is the following: when seeking estimates of $\ell$th order, the moment loss can be avoided in the highest order energy estimates (i.e., the estimates of $\partial_x^\alpha \partial_v^\beta f$ with $|\alpha|+|\beta| = \ell$) by carefully exploiting the symmetry properties of $Q$.  However, intermediate terms do not have such nice symmetry properties and the moment loss can be handled by trading regularity for decay via interpolation (see \Cref{l:moment_interpolation}).  This uses the fact that we have {\em more} moments on the $0^{\text{th}}$ order term (reflected in the fact that $m\gg n$ and $f\in L^{\infty,m}$), and allows us to obtain a closed estimate for $f$ in $Y_T^{k,n,m}$.  A similar idea was used by He and Yang to construct polynomially bounded solutions to the Landau-Coulomb equation via cut-off Boltzmann equation~\cite{he2014boltzmannlandau}.  We also mention the work of Luk~\cite{luk2019vacuum}, who uses a mixture of $L^2$ and $L^\infty$ spaces to prove stability of vacuum (the steady state $0$) for the Landau equation.  The choice of space for Luk is crucial in obtaining appropriate time decay (see also the work of Chaturvedi~\cite{chaturvedi2019local,chaturvedi2020stability} who sidesteps this with a ``hierarchy of weighted norms'' involving fewer weights on higher regularity norms).


To pass from \emph{a priori} estimates to an existence proof, we first construct a sequence $f_i$ with $f_0(t,x,v) = f_{\rm in}(x,v)$ and
\begin{equation}\label{e:linear-iteration}
 \begin{cases} \partial_t f_i+ v\cdot \nabla_x f_i &= Q(f_{i-1},f_i),\\
f_i(0,\cdot,\cdot) &= f_{\rm in}.\end{cases}
\end{equation}
To prove the existence of $f_i$, we use the method of continuity combined with our \emph{a priori} estimates. 
With the existence of $f_i$ established, we can build a solution to the true nonlinear equation \eqref{e:boltzmann} with a compactness argument based on the same \emph{a priori} estimates.  A benefit of our construction is that we obtain the non-negativity of $f_i$ for free (see \Cref{p:linear_existence}).

It is interesting to compare our method to the overall strategy applied in earlier works such as  \cite{amuxy2010regularizing,amuxy2011bounded, amuxy2013mild,morimoto2015polynomial}. This strategy, which has been quite successful in constructing short-time solutions to the non-cutoff Boltzmann equation in various regimes, is based on approximating \eqref{e:boltzmann} with the cutoff Boltzmann equation, by replacing $b(\cos\theta)$ in \eqref{e:B} with a bounded function $b_\eps(\cos\theta)$ that converges to $b(\cos\theta)$ as $\eps \to 0$. 
Letting $Q_\eps(f,f)$ be the corresponding collision operator, this allows one to write $Q_\eps$ as a sum of gain and loss terms, $Q_\eps = Q_+ - Q_-$, with
\[\begin{split}
 Q_+(f_1,f_2) &= \int_{\R^3}\int_{\mathbb S^{2}} f_1(v_*')f_2(v') |v-v'|^\gamma b_\eps(\cos\theta)  \dd \sigma \dd v_*,\\ 
 Q_-(f_1,f_2) &= \int_{\R^3}\int_{\mathbb S^{2}}  f_1(v_*)f_2(v)|v-v'|^\gamma b_\eps(\cos\theta)  \dd \sigma \dd v_*.
 \end{split}
 \]
Note that these integrals may not be well-defined if the singularity of $b(\cos\theta)$ is included. Then, the authors of \cite{morimoto2015polynomial}  set up an iteration of the form
\[ \begin{cases} \partial_t f_i+ v\cdot \nabla_x f_i &= Q_+(f_{i-1},f_{i-1}) - Q_-(f_{i-1},f_{i}),\\
f_i(0,\cdot,\cdot) &= f_{\rm in},\end{cases}\]
which is different from our iteration \eqref{e:linear-iteration} because the right-hand side is not the linear  operator $Q(f_{i-1},f_i)$. This approach has the benefit that existence of $f_i$ can be deduced easily from a Duhamel-type formula. Weighted Sobolev estimates on $f_i$ provide convergence to a solution $f_\eps$ of the nonlinear (cutoff) Boltzmann equation with right-hand side $Q_\eps(f_\eps,f_\eps)$, and finally, estimates that are uniform in $\eps$ establish existence of a solution to the non-cutoff equation by compactness.


One benefit of working with the cutoff equation is that when $s \geq 1/2$ the non-cutoff collision operator $Q$ is well-defined only in a ``principal value" sense because of the singularity in $b(\cos\theta)$ and, in principle, integral estimates on $Q$ may not commute with the limit involved in this principal value.  In previous well-posedness proofs, this potential issue is sidestepped since estimates are obtained at the cutoff level with bounded angular cross-section $b_\eps(\cos\theta)$. However, despite working at the non-cutoff level, this issue causes no problems in our estimates. Indeed, for the only estimates in which we work with $Q$ directly (\Cref{p:Q_estimates}.(iii)), our manipulations of $Q$ follow the work of~\cite{amuxy2011bounded} in which the authors are working directly with the cutoff kernel (although the bounds we obtain are slightly different). On the other hand, when we work with the Carleman decomposition (see ~\Cref{l:kernel_bounds}), our estimates are always done by decomposing integrals into sums of integrals over annuli, and this is compatible with the principal value. Thus, we ignore this technical point in the sequel.

The main novelty of \cite{morimoto2015polynomial} was the ability to estimate $f_i$ directly, rather than dividing the solution by a time-dependent Gaussian $\mu = e^{-(\rho+\kappa t)\vv^2}$ and working with the equation for $g = \mu^{-1} f$ as in \cite{amuxy2010regularizing,amuxy2011bounded, amuxy2013mild}, so that the assumption of Gaussian decay in velocity in the initial data can be removed.

Our method in this article is novel in that we base our approximation on the linear, non-cutoff Boltzmann equation, so that no cutoff approximation is required. It is also worth noting that essentially the same \emph{a priori} estimates in $L^{\infty,m}$ and $H^{k,n}$ are used to construct the linear and nonlinear solutions.  This is in contrast to previous works in which separate estimates are required at the linear (and cutoff) and the non-linear (and non-cutoff) levels.

On a technical level, we make heavy use of the decomposition of $Q(f,f)$ into a fractional diffusion operator $Q_{\rm s}(f,f)$ and a lower-order term $Q_{\rm ns}(f,f)$ (see Section \ref{s:carleman}), taking inspiration from \cite{imbert2018decay,  imbert2019smooth, imbert2016weak, silvestre2016boltzmann} and using some of their estimates out of the box. We also borrow some Sobolev estimates on $Q(g,f)$ from the work of the AMUXY group such as \cite{amuxy2010regularizing, amuxy2011qualitative, amuxy2011bounded} (see Section \ref{s:sobolev}).  Whenever possible, we rely on these existing estimates; however, they are not on their own sufficient for the proof of \Cref{t:main}.  Hence, our proof requires new estimates on the collision operator in our mixed space $X^{k,n,m}$.

The restriction $\gamma > \max\{-3,-3/2-2s\}$ in our results comes from the use of the estimates of \cite{amuxy2010regularizing, amuxy2011qualitative, amuxy2011bounded}. (See in particular Theorem \ref{t:amuxy-sobolev-v} below.) The restriction $\gamma < 0$ comes from our estimate of the nonlocal diffusion term $Q_s$, where the pointwise decay from $|v-v'|^\gamma$ is needed to control the tail of the integral (see Step Four of the proof of Proposition \ref{p:Q_estimates}(i) (Section \ref{s:Qs-proof}).

\subsection{Outline of the paper}

In \Cref{s:prelim}, we compile known results and relatively easy-to-deduce estimates on the collision operator and on interpolations between Sobolev spaces with weights.  In \Cref{s:main}, we state our main estimates (\Cref{p:Q_estimates}) on the collision operator $Q$, and then use these estimates to construct solutions.  These estimates are what take up the bulk of the effort in this proof and are established in \Cref{s:Q}.   Finally, we include in the appendix an inequality that relates integrals over somewhat complicated geometric spaces with their (weighted) $L^1$ norm.

\section{Preliminaries and known results}\label{s:prelim}


\subsection{The Carleman representation of $Q$}\label{s:carleman}

We make use of the Carleman representation of $Q$; that is, by adding and subtracting $f_1(v_*')f_2(v)$ inside the integral, we write
\[
	Q(f_1, f_2)
		= Q_{\rm s}(f_1,f_2) + Q_{\rm ns}(f_1,f_2),
\]
with 
\[ 
\begin{split}
Q_{\rm s}(f_1,f_2) &:= \int_{\R^3} \int_{\mathbb S^2} (f_2(v') - f_2(v)) f_1(v_*') B(|v-v_*|,\cos\theta) \dd \sigma \dd v_*,\\
Q_{\rm ns}(f_1,f_2) &:= f_2(v) \int_{\R^3}\int_{\mathbb S^{2}} (f_1(v_*') - f_1(v_*)) B(|v-v_*|,\cos\theta) \dd \sigma \dd v_*.
\end{split}
\]
Here, ``s'' stands for singular, that is, the smoothing operator, and ``ns'' stands for nonsingular, that is, the lower order multiplication operator term. These names are justified by the following re-formulations of $Q_{\rm s}$ and $Q_{\rm ns}$. The formula for $Q_{\rm s}$ used below was established by Silvestre in \cite{silvestre2016boltzmann}, using Carleman's change of variables \cite{carleman1933boltzmann}. (See also \cite{alexandre20003d} for earlier, related formulas.) The form of $Q_{\rm ns}$ (Lemma \ref{l:Qns} below) essentially follows from the cancellation lemma of  \cite{alexandre2000entropy}.
\begin{lemma}\label{l:Qs}{\cite[Lemma 4.1]{silvestre2016boltzmann}}
The term $Q_{\rm s}(f_1,f_2)$ can be written
\begin{equation}
Q_{\rm s}(f_1,f_2)(v)
		= \pv \int_{\R^3} K_{f_1}(v,v')[f_2(v') - f_2(v)]\dd v',
\end{equation}
where ``$\pv$'' denotes the principal value and the kernel $K_{f_1}(v,v')$ is defined by
\begin{equation*}
K_{f_1}(v,v') := \frac{1}{|v-v'|^{3 + 2s}} \int_{(v'-v)^\perp} f_1(v+w)|w|^{\gamma + 2s + 1} \tilde b(\cos(\theta)) \dd w,
\end{equation*}
and $\cos(\theta/2) = |w|/\sqrt{|v-v'|^2 + |w|^2}$.
\end{lemma}
Note that here, we used the assumption~\eqref{e:B} to write $b(\cos(\theta)) = |\theta|^{-2-2s} \tilde b(\cos(\theta))$ for a bounded, positive function $\tilde b$. In general, we abuse notation and omit the ``$\pv$'' in our notation below (see the discussion in \Cref{s:ideas}). We have the following estimates for $K_{f_1}$:


\begin{lemma}\label{l:kernel_bounds}
For any measurable $h: \R^3 \to \R$, $r>0$, and $v\in \R^3$,
\[\begin{split}
	&\int_{B_{2r}(v) \setminus B_r(v)} |K_h(v,v')| \dd v'
		\lesssim r^{-2s} \left( \int_{\R^3} |h(z)| |v-z|^{\gamma + 2s} \dd z\right) \qquad \text{ and}\\
	&\int_{\R^3 \setminus B_r(v')} K_h(v,v') \dd v
		\lesssim r^{-2s} \left( \int_{\R^3\setminus B_r(v)} |h(z)| |z-v'|^{\gamma + 2s} \dd z\right).
\end{split}\]
Also, denoting $K_h' = K_h(v',v)$ and fixing any $\e > 0$, we find, for any $\alpha \in [0,1]$,
\[
	\left|\pv \int_{\R^3} (K_h  - K_h') \dd v'\right|
		\lesssim \int_{\R^3} |h(z)||v - z|^{\gamma} \dd z,
\]
and, for all $v, v' \in \R^3$,
\[
	|K_h - K_h'|
		\lesssim  \frac{\|\langle\cdot\rangle^{(2 + (\gamma+2s + 1)_+ +\e} h\|_{C^\alpha}}{|v-v'|^{3 + 2s - \alpha}}\left(\vv^{\gamma + 2s + 1} + \vvp^{\gamma + 2s + 1}\right).
\]
\end{lemma}
\begin{proof}
The first three inequalities are simply \cite[Lemmas 3.4 and 3.5]{imbert2016weak}.  The fourth inequality follows by the observation that
\[
	K_h - K_h'
		= |v-v'|^{-(3 + 2s)} \int_{(v-v')^\perp} (h(\xi - v) - h(\xi - v')) |\xi|^{\gamma + 2s + 1} \tilde b(\cos(\theta)) \dd \xi.
\]
\end{proof}

For the lower-order term $Q_{\rm ns}$, we have:
\begin{lemma}\label{l:Qns}
{\cite[Lemmas 5.1 and 5.2]{silvestre2016boltzmann}}
The term $Q_{\rm ns}(f_1,f_2)$ can be written
\begin{equation}
Q_{\rm ns}(f_1,f_2) = f_2 (S* f_1), 
\end{equation}  
where $S(v) := C_S |v|^\gamma$ and $C_{S}>0$ depends only on $s$.
\end{lemma}

\subsection{Sobolev estimates on $Q(g,f)$.}\label{s:sobolev}

\begin{theorem}\cite[Proposition 2.9]{amuxy2011qualitative}  \label{t:amuxy-sobolev-v}
Let $0< s< 1$ and $0 > \gamma > \max\left\{-3,-\frac 3 2 - 2s\right\}$. Then for any $n \in \R$ and $k \in [-1,0]$, there holds
\[
	\|Q(g,f)\|_{H^{k,n}_v(\R^3)} \lesssim \left( \|g\|_{L^{1,n_+ +(\gamma+2s)_+}_v(\R^3)} + \| g \|_{L^2_v(\R^3)} \right)
		\|f\|_{H^{\max\left\{k+2s,(2s-1+\e)_+\right\},(n+\gamma+2s)_+}_v(\R^3)},
\]
for any $f$ and $g$ such that the right-hand side is finite.
In addition, we have that for any $\theta \in [-1,0]$
and any $g$, $f$, $h$, that
\[
	\Big|\int Q(g,f) h \dd v \Big|
		\lesssim \left( \|g\|_{L^{1,(\gamma+2s)_+}(\R^3)} + \| g \|_{L^2_v(\R^3)} \right)
			\|f\|_{H^{\max\{\theta+2s,(2s-1+\e)_+\}, (\gamma+2s)_+}(\R^3)} \|h\|_{H^{-\theta}(\R^3)}.
\]
\end{theorem}
We note that the second estimate in \Cref{t:amuxy-sobolev-v} follows from \cite[Proposition 2.1]{amuxy2011qualitative}, which proves the estimate
for the collision kernel localized to the low relative velocity regime ($|v-v'| \lesssim 1$), and \cite[Theorem 2.1]{amuxy2010regularizing}, which handles the remainder ($|v-v'| \gtrsim 1$).  

Finally, we state an estimate on the commutator of weights and the collision operator.

\begin{proposition}\cite[Proposition 2.8]{amuxy2011qualitative}\label{p:commutator}
Let $0 < s < 1$ and $0\geq\gamma$.  For any $\ell \in \R$ and $\e>0$,
	\[
		\int h(\vv^\ell Q(f,g) - Q(f,\vv^\ell g)) \dd v
			\lesssim \|f\|_{L^{2,\ell + 3/2 + (2s-1)_+ + \e}}\|g\|_{H^{(2s-1+\e)_+,\ell + (2s-1)_+}} \|h\|_{L^2}.
	\]
\end{proposition}

\subsection{Interpolation}



\begin{lemma}\label{l:moment_interpolation}
	Fix $n \geq 0$ and $m \geq 0$.  Suppose that $f \in L^{\infty,m}\cap H^{k,n}(\R^3)$ and $k' \in (0,k)$.  Then if $\ell < (m-3/2)(1- k'/k) + n(k'/k)$, we have
	\[
		\|f\|_{H^{k',\ell}}
			\lesssim \|f\|_{L^{\infty,m}}^{1 - \frac{k'}{k}} \|f\|_{H^{k,n}}^{\frac{k'}{k}}
	\]
\end{lemma}
\begin{proof}
Notice that
\[
	\|f\|_{H^{k',\ell}}^2
		\approx \sum_{z \in \Z^3} \|\vv^\ell  f\|_{H^{k'}(B_2(z))}^2
		\approx \sum_{z\in\Z^3} \langle z \rangle^{2\ell} \| f\|_{H^{k'}(B_2(z))}^2.
\]
Fix $\epsilon>0$ to be chosen.  Let $m' < m-3/2$ be such that $\ell = m' (1-k'/k) + n(k'/k)$.    Using standard interpolation between Sobolev spaces, we find
\[\begin{split}
	\|f\|_{H^{k',\ell}}^2
		&\lesssim \sum_{z\in\Z^3} \langle z \rangle^{2\ell} \| f\|_{L^{2}(B_2(z))}^{2\left(1 - \frac{k'}{k}\right)}\| f\|_{H^{k}(B_2(z))}^\frac{2k'}{k}
		\lesssim \sum_{z\in\Z^3} \langle z \rangle^{2m'\left(1 - \frac{k'}{k}\right) + \frac{2nk'}{k}} \| f\|_{L^{\infty}(B_2(z))}^{2\left(1 - \frac{k'}{k}\right)}\| f\|_{H^{k}(B_2(z))}^\frac{2k'}{k}\\
		&\approx \sum_{z\in\Z^3}  \|\vv^{m'} f\|_{L^{2}(B_2(z))}^{2\left(1 - \frac{k'}{k}\right)}\|\vv^n f\|_{H^{k}(B_2(z))}^\frac{2k'}{k}\\
		&\lesssim \sum_{z\in \Z^3} \big(\e^{\frac{k}{k-k'}}  \|\vv^{m'} f\|_{L^2(B_2(z))}^2 + \e^{-k'/k} \| f\|_{H^{k,n}( B_2(z))}^2\big)
		\approx \e^{\frac{k}{k-k'}}  \|\vv^{m'} f\|_{L^2}^2 + \e^{-k'/k} \| f\|_{H^{k,n}}^2.
\end{split}\]
Optimizing in $\e$ leads to $\|f\|_{H^{k',\ell}} \lesssim \| f\|_{L^{2,m'}}^{1 - \frac{k'}{k}}  \| f\|_{H^{k,n}}^\frac{k'}{k}.$ 
The proof is complete as $\|f\|_{L^{2,m'}} \lesssim \|f\|_{L^{\infty,m}}$.
\end{proof}

\section{Proof of the main theorem}\label{s:main}

In order to establish our main theorem, we first construct a solution to the linear problem
\[
	(\partial_t + v\cdot \nabla_x) f
		= Q(g,f)
\]
for a fixed $g$.  We then use an iteration argument to find a fixed point where $g=f$.  In both steps we require estimates on $Q$, which are stated in the following subsection.

\subsection{Main estimates on $Q$}

The bulk of our argument is in estimating various quantities involving $Q$.  We state these estimates now, but their proof is in \Cref{s:Q}.

\begin{proposition}\label{p:Q_estimates}
	Fix any $\e>0$,  let $\eta = 3/2 + (\gamma+2s)_+ + \e$, and suppose that $n > \eta$.  We have the following estimates:
	\begin{enumerate}[(i)]
		\item Suppose that $m > n + 3/2 + \gamma + \e$, $\alpha \in [0,1)\cap (2s-1,1)$, and $f,g: \R^3 \to \R$.  Then
			\begin{equation*}
			\|\Qs(g,f)\|_{L^{2,n}}
				\lesssim \left(\|f\|_{L^{\infty,m}}
						+ \|\vv^{n + \eta + 2} D_v f\|_{C^\alpha}\right)
					\|g\|_{L^{2,n}}.
			\end{equation*}
			
		\item If $f,g: \R^3 \to \R$,
			\[
				\|\Qns(g,f)\|_{L^{2,n}}
					\lesssim \|f\|_{L^{\infty, n + \eta}}\|g\|_{L^{2,n}}.
			\]
		\item If  $f: \R^3 \to \R$ and $g: \R^3 \to [0,\infty)$ and $m>2n + 3 + \gamma$,
			\[
				\int \vv^{2n} f Q(g,f) \dd v
					\lesssim \| f \|_{L^{2,n}}^2 \| g \|_{L^{\infty,m}}.
			\]
		\item Let $\partial = \partial_{x_i}$ or $\partial_{v_i}$ for some $i = 1,2,3.$  Then, if $\alpha \in ((2s-1)_+,s)$ and  $f,g: \T^3\times\R^3 \to \R$, 
			\[\begin{split}
				\int \vv^{2n} Q(g, f) \partial f \dd v \dd x
					&\lesssim \left(\|g\|_{L^\infty_x L_v^{2,n + \eta + (2s-1)_+}} + \|\vv^{5/2 + \eta} g\|_{L^\infty_xC_v^\alpha} + \|\partial g\|_{L^\infty_x L_v^{2,\eta}}\right)\\
						&\qquad \cdot \left(\|f\|^2_{H^{s, n + 3/2 + \eta}} + \|f\|^2_{H^{1,n}}\right).
			\end{split}\]
		\item If $m > 3 + \gamma + 2s$ and $f = \vv^{-m}$, then
			\[
				\|Q(g,f)\|_{L^{\infty,m}}
					\lesssim \|g\|_{L^{\infty,m}}.
			\]
	\end{enumerate}
\end{proposition}

The estimates above are not necessarily sharp (particularly in the weights required); however, stating the sharp form obtained in the sequel is cumbersome and unnecessary for our main well-posedness result.  We also note that \Cref{p:Q_estimates}.(ii) is a consequence of \cite[Lemma 2.1]{HST2018landau}, so we omit its proof (indeed, we can write $Q(g,f) \approx \bar c[g]f$ in the notation of~\cite{HST2018landau}).  The remaining estimates are proved in \Cref{s:Q}.

To our knowledge, with the exception of \Cref{p:Q_estimates}.(ii), these estimates are new, and existing estimates do not suffice for our purposes.  In general existing estimates, such as those found in~\cite{amuxy2010regularizing, amuxy2011qualitative, amuxy2011bounded} cover or are tailored to settings where $f$ has Gaussian decay.

\subsection{Solving the linear equation}


In order to show that the linear problem has a solution, we use the method of continuity.  In order to do this, we work with a regularized equation that we formulate now.  Let $\chi: \R^3\to \R$ be a smooth-cutoff function such that $0 \leq \chi \leq 1$, $\chi \equiv 1$ on $B_{1/2}$, $\chi \equiv 0$ on $B_1^c$, and $\int \chi(v) \dd v = 1$.  For any $\e,\delta\geq 0$, we define the following quantities.  First, for any $\phi : \T^3\times \R^3 \to \R$, let
\[
	\phi^\e(x,v) = \frac{1}{\e^6} \int \chi\left( \frac{x-y}{\e}\right) \chi\left( \frac{v-w}{\e}\right) \phi(y,w) \dd y \dd w \qquad \text{ if } \e >0,
\]
and $\phi^\e = \phi$ if $\e = 0$.  Then, let the regularized collision operator be defined by
\[
	Q_{\e,\de}(g(x,\cdot),f(x,\cdot))(v)
		= \chi(\de v) Q(g^\e(x,\cdot), \chi(\de \cdot) f)
		\qquad \text{ for any } (x,v) \in \T^3\times \R^3.
\]
Finally, for any $\sigma \in [0,1]$, we define the differential operator
\[
	\cL_{\sigma, \e,\de}(f)
		= \partial_t f + \sigma \chi(\de v) v \cdot \nabla_x f - (\e + (1-\sigma)) \Delta_{x,v}f - \sigma Q_{\e,\de}(g,f).
\]

Here, $\sigma$ is the parameter in the method of continuity; it connects the linear Boltzmann equation and the heat equation.  The parameter $\e$ smoothes the initial data and $g$ and provides additional coercivity.  Its purpose is to allows us to work in a smooth setting where we can apply Schauder estimates up to $\{t=0\}$.  The parameter $\delta$ cuts off large velocities and allows us to sidestep issues with moment generation when we obtain estimates in order to apply the method of continuity.  The necessity for two regularization parameters is technical and has to do with the fact that the important estimate~\eqref{e:a_priori2} may not hold for $\delta>0$ (case four of its proof relies on certain symmetries that are broken when $\de>0$).

We now establish \emph{a priori} estimates that hold for both the full equation and the regularized one above. This is done in the following proposition.

\begin{proposition}\label{p:a_priori}
	Suppose that $T>0$, $k \geq 5$, $n> 3/2 + (\gamma+2s)_+$, $\sigma \in [0,1]$, $\e,\de \geq 0$, and $m \geq 0$.  Suppose that $R,f \in Y_T^{k,n,m}$ for some $m$ and satisfy
	\begin{equation}\label{e:linear_interpolation}
	\begin{cases}
		\cL_{\sigma, \e, \delta} f = R
			\qquad &\text{ in } (0,T) \times \T^3\times\R^3,\\
		f(0,\cdot,\cdot) = f_{\rm in} 
			\qquad &\text{ on } \T^3\times \R^3.
	\end{cases}
	\end{equation}
	For any $\mu>0$, if $\delta = 0$ and $m \geq 3/2 + \mu$ or if $\delta > 0$, then, for all $t_0 \in [0,T]$,
	\begin{equation}\label{e:a_priori1}
	\begin{split}
		\|f\|_{L^{\infty,m}([0,t_0]\times\T^3\times\R^3)}
			\leq &\exp\Big\{C\int_0^{t_0} \|g(t)\|_{L^{\infty,\max\{m,3/2 + \mu\}}} \dd t\Big\}\\
				&\qquad \Big( \|f_{\rm in}\|_{L^{\infty,m}(\T^3\times\R^3)}	+  \int_0^{t_0} \|R(t)\|_{L^{\infty,m}(\T^3\times\R^3)} \dd t\Big).
	\end{split}
	\end{equation}
	If $\de = 0$ and $m$ is sufficiently large depending on $k$, $n$, $\gamma$, and $s$,
	\begin{equation}\label{e:a_priori2}
	\begin{split}
		\|f\|^2_{L^\infty_t H_{x,v}^{k,n}([0,t_0]\times \T^3\times \R^3)}
			&\leq \exp\Big\{C\int_0^{t_0} \|g\|_{L^\infty_t X^{k,n,m}_{x,v}([0,t]\times\T^3\times\R^3)} \dd t\Big\}\\
			&\qquad  \Big( \|f_{\rm in}\|^2_{X^{k,n,m}} + C t_0 \|f_{\rm in}\|^2_{L^{\infty,m}}
				+ C\int_0^{t_0} \|R(t)\|^2_{X^{k,n,m}(\T^3\times\R^3)} \dd t\Big).
	\end{split}
	\end{equation}
	Here $C$ is a universal constant depending only on $T$, $\gamma$, $s$, $k$, $n$, and $m$.  If $m \leq 3/2$, it additionally depends on $\delta$ in~\eqref{e:a_priori1}.
\end{proposition}
\begin{proof}
The estimates are exactly the same when $\e>0$ so, to make the notation less cumbersome, we only consider the case $\e =0$.  
We first establish the estimate on the $L^{\infty,m}$-norm of $f$.

\medskip

{\bf The $L^{\infty,m}$ bound, \eqref{e:a_priori1}.} The case $\delta >0$ is simpler than the case $\de = 0$ so, we consider only the case $\de = 0$ and omit the case $\de >0$.  It is enough to construct a super-solution.  Indeed, fix $t_0$ and let 
	\[
		A(t) = C \|g(t)\|_{L^{\infty,m}(\T^3 \times \R^3)},
	\]
	for $C$ to be determined, and
	\[
		\overline f(t,x,v) = e^{\int_0^t A(s) \dd s} \left(\|f_{\rm in}\|_{L^{\infty,m}} \vv^{-m} + \int_0^t \|R(s) \|_{L^{\infty,m}} \dd s\right).
	\]
	Clearly, $f_{\rm in} \leq \overline f(0)$.  Hence, by the comparison principle, we are finished if we show that
	\begin{equation}\label{e:c1}
		(\partial_t + \sigma v\cdot\nabla_x) \overline f
			- \sigma Q(g,\overline f) - (1-\sigma) \Delta_{x,v} \overline f - R \geq 0.
	\end{equation}
	This follows easily from \Cref{p:Q_estimates}.(v).  Indeed,
	\[
		(\partial_t + \sigma v\cdot\nabla_x) \overline f
			= A \overline f + e^{\int_0^t A(s) \dd s} \|R(t)\|_{L^{\infty,m}},
	\]
	and from \Cref{p:Q_estimates}.(v), we have $|Q(g,\overline f)|
			\lesssim \|g(t)\|_{L^{\infty,m}} \overline f.$ 
	In addition, a direct computation yields $\Delta_{x,v} \overline f \lesssim \overline f.$  Letting $C$ be the sum of these two implied constants, we obtain~\eqref{e:c1}.  Thus, $f \leq \overline f$ by the comparison principle.  This concludes the proof of~\eqref{e:a_priori1}.

\medskip

{\bf The $H^{k,n}$ bound, \eqref{e:a_priori2}.}    
Let $\alpha, \beta \in \N_0^3$ be any multi-indices such that $|\alpha| + |\beta| = k$.  Then, differentiating \eqref{e:linear_interpolation}, one has
\begin{equation}\label{e:differentiated_landau}
	\begin{split}
		\partial_x^\alpha \partial_v^\beta f_t
		&+ \sigma v\cdot \nabla_x \partial_x^\alpha \partial_v^\beta f
		+ \sigma \sum_{i=1}^3 \beta_i \partial_{x_i}\partial_x^\alpha\partial_v^{\beta - e_i} f\\
		&= \sigma \sum_{\substack{\alpha' + \alpha'' = \alpha,\\ \beta' + \beta'' = \beta}} Q( \partial_x^{\alpha'}\partial_v^{\beta'} g, \partial_x^{\alpha''} \partial_v^{\beta''} f)
			+(1-\sigma) \Delta_{x,v} \partial_x^\alpha \partial_v^\beta f
			+ \partial_x^\alpha \partial_v^\beta R.
	\end{split}
\end{equation}
Here $e_1 = (1,0,0)$, $e_2 = (0,1,0)$, and $e_3 = (0,0,1)$, and by $\partial_x^\alpha$ we mean $\partial_{x_1}^{\alpha_1} \partial_{x_2}^{\alpha_2}\partial_{x_3}^{\alpha_3}$.  The terms $\partial_v^\beta$ are defined similarly.

Fix any $t_0>0$.  Multiplying~\eqref{e:differentiated_landau} by $\vv^{2n} \partial_x^\alpha \partial_v^\beta f$ and integrating in $x$ and $v$ with $t$ fixed (for the remainder of this proof, all $x$ integrals are over $\mathbb T^3$, and all $v$ integrals are over $\R^3$), we find
\begin{equation}\label{e:linearized_energy}
\begin{split}
	\frac{1}{2} \ddt &\int |\vv^n \partial_x^\alpha \partial_v^\beta f|^2 \dd x \dd v
		= - \sigma\int \left(\sum_{i=1}^3 \beta_i \partial_{x_i}\partial_x^\alpha\partial_v^{\beta - e_i} f\right) \vv^{2n} \partial_x^\alpha \partial_v^\beta f \dd x \dd v \\
			&\qquad+ \sigma \sum_{\substack{\alpha' + \alpha'' = \alpha,\\ \beta' + \beta'' = \beta}} \int Q( \partial_x^{\alpha'}\partial_v^{\beta'} g, \partial_x^{\alpha''} \partial_v^{\beta''} f) \vv^{2n}  \partial_x^\alpha \partial_v^\beta f \dd x \dd v\\
			&\qquad- (1-\sigma) \int \left| \nabla_{x,v} \partial_x^\alpha \partial_v^\beta f\right|^2 \dd x \dd v
			+ \int  \partial_x^\alpha \partial_v^\beta R \vv^{2n} \partial_x^\alpha \partial_v^\beta f \dd x \dd v.
\end{split}
\end{equation}
The first term on the right hand side is clearly bounded by $\|f\|_{H^{k,n}}^2$, the third term has a good sign, and the last term is clearly bounded by $\|R(t)\|_{H^{k,n}}\|f(t)\|_{H^{k,n}} \lesssim \|R(t)\|_{H^{k,n}}^2 + \|f(t)\|_{H^{k,n}}^2$.  Hence, we focus our attention on the second term on the right hand side.  We claim that, for any $t$,
\begin{equation}\label{e:Q_key_estimate}
	\sum_{\substack{\alpha' + \alpha'' = \alpha,\\ \beta' + \beta'' = \beta}} \int Q( \partial_x^{\alpha'}\partial_v^{\beta'} g(t), \partial_x^{\alpha''} \partial_v^{\beta''} f(t)) \vv^{2n} \partial_x^\alpha \partial_v^\beta f \dd v \dd x
		\lesssim \|g(t)\|_{X^{k,n,m}} \|f(t)\|_{X^{k,n,m}}^2.
\end{equation}

We show how to conclude assuming that~\eqref{e:Q_key_estimate} is proved.  In this case,
\begin{equation}\label{e:energy_inequality1}
	\frac{1}{2} \ddt \int |\vv^n \partial_x^\alpha \partial_v^\beta f(t)|^2 \dd x \dd v
		\lesssim (\|g(t)\|_{X^{k,n,m}}+1)\|f(t)\|_{X^{k,n,m}}^2 + \|R(t)\|_{H^{k,n}}^2.
\end{equation}
Using the definition of $X$, we have
$
	\|f(t)\|^2_{X^{k,n,m}}
		= \|f(t)\|^2 _{H_{x,v}^{k,n}} + \|f(t)\|^2_{L^{\infty,m}}. 
$ 
Plugging this into~\eqref{e:energy_inequality1}, we find
\[\begin{split}
	\ddt \|f(t)\|^2_{H^{k,n}}
		\lesssim (\|g(t)\|_{X^{k,n,m}}+1)(\|f(t)\|^2 _{H_{x,v}^{k,n}} + \|f(t)\|^2_{L^{\infty,m}}) 
			+ \|R(t)\|^2_{H^{k,n}}.
\end{split}\]

With the above inequality, we can apply Gr\"onwall's inequality to conclude that 
\[
	\|f(t_0)\|_{H^{k,n}}^2
		\leq \exp\Big\{ C \int_0^{t_0} \|g(t)\|_{X^{k,n,m}} \dd t + Ct_0 \Big\}\Big( \|f_{\rm in}\|^2_{H^{k,n}} + C\int_0^{t_0} (\|R(t)\|^2_{H^{k,n}} + \|f(t)\|_{L^{\infty,m}}^2) \dd t\Big).		
\]
Then~\eqref{e:a_priori2} is established by  using~\eqref{e:a_priori1}.  Thus, the proof is concluded after establishing~\eqref{e:Q_key_estimate}.

\medskip

{\bf The proof of \eqref{e:Q_key_estimate}.} 
There are four major cases to consider, depending on the size of $\alpha'$, $\alpha''$, $\beta'$ and $\beta''$.  In each case, we consider $t$ to be fixed and omit it notationally.

One inequality that is used throughout is that, for any $\e>0$, $\ell\in \R$, and $h \in L^{2,\ell + 3/2+\e}$,
\begin{equation}\label{e:L1_L2}
	\|h\|_{L^{1,\ell}} \lesssim \|h\|_{L^{2,\ell + 3/2 + \e}}.
\end{equation}

\medskip

 {\bf Case one: $|\alpha'| + |\beta'| \leq k-1$ and $|\alpha''| + |\beta''| \leq k-2$.} 
First, we apply Cauchy-Schwarz to find
\[
	\int Q( \partial_x^{\alpha'}\partial_v^{\beta'} g, \partial_x^{\alpha''} \partial_v^{\beta''} f) \vv^{2n} \partial_x^\alpha \partial_v^\beta f \dd x \dd v
		\leq \|f\|_{H^{k,n}} \left(\int |\vv^n Q( \partial_x^{\alpha'}\partial_v^{\beta'} g, \partial_x^{\alpha''} \partial_v^{\beta''} f)|^2 \dd x \dd v\right)^{1/2}.
\]
From \Cref{t:amuxy-sobolev-v} and~\eqref{e:L1_L2}, we find that
\begin{equation}\label{e:QL2n}
\begin{split}
	\|Q(\partial_x^{\alpha'}\partial_v^{\beta'} g, & \partial_x^{\alpha''} \partial_v^{\beta''} f)\|_{L^{2,n}}^2
	= \int \| Q(\partial_x^{\alpha'}\partial_v^{\beta'} g, \partial_x^{\alpha''} \partial_v^{\beta''} f) \|_{L^{2,n}_v}^2 \dd x \\
		&\lesssim \int \left(\|\partial_x^{\alpha'}\partial_v^{\beta'} g\|_{L_v^{1,(\gamma+2s)_+}}^2
			+ \| \partial_x^{\alpha'}\partial_v^{\beta'} g \|_{L^2_v}^2 \right)
			\| \partial_x^{\alpha''} \partial_v^{\beta''} f\|_{H_v^{2s, n+\gamma + 2s}}^2 \dd x \\
		&\lesssim \int \|\partial_x^{\alpha'}\partial_v^{\beta'} g \|_{L_v^{2,(\gamma+2s)_+ +3/2 + \epsilon}}^2
			\| \partial_x^{\alpha''} \partial_v^{\beta''} f\|_{H_v^{2s, n+\gamma + 2s}}^2 \dd x.
\end{split}
\end{equation}
If $|\alpha'| + |\beta'| \leq k-2$, then we continue~\eqref{e:QL2n} with H\"older's inequality to get
\[
	\|Q(\partial_x^{\alpha'}\partial_v^{\beta'} g, \partial_x^{\alpha''} \partial_v^{\beta''} f)\|_{L^{2,n}}^2
		\lesssim \|\partial_x^{\alpha'}\partial_v^{\beta'} g \|_{L^\infty_x L^{2, (\gamma+2s)_+ +3/2+\epsilon}_v}^2
			\| \partial_x^{\alpha''} \partial_v^{\beta''} f\|_{H^{2s, n+\gamma + 2s}}^2.
\]
Using \Cref{l:moment_interpolation}, which we may apply as long as $m$ is sufficiently large because
$|\alpha''| + |\beta''| +2s \leq k - 2(1-s) < k$, we obtain
\[
\begin{split}
	\|Q(\partial_x^{\alpha'}\partial_v^{\beta'} g, \partial_x^{\alpha''} \partial_v^{\beta''} f)\|_{L^{2,n}}
		&\lesssim \|g\|_{L^\infty_x H^{k-2,(\gamma+2s)_+ + 3/2 + \e}_v} 
			\|f\|_{H^{k - 2(1-s), n+\gamma + 2s}}
		\lesssim \|g\|_{X^{k,n,m}} \|f\|_{X^{k,n,m}}.
\end{split}
\]
The last inequality is by the Sobolev embedding $H^2 \subset L^\infty$.

If, on the other hand, $|\alpha'| + |\beta'| = k-1$ (so that $\partial_x^{\alpha''} \partial_v^{\beta''} f = \partial f$),
then we continue~\eqref{e:QL2n} as follows.  Applying H\"older's inequality, we find
\[\begin{split}
	\|Q(\partial_x^{\alpha'}\partial_v^{\beta'} g, \partial_x^{\alpha''} \partial_v^{\beta''} f)\|_{L^{2,n}}^2
		&\lesssim \| g \|_{H^{k-1, (\gamma+2s)_++3/2+\e}}^2
			\| \partial f \|_{L^\infty_x H^{2s, n+\gamma+2s}}^2
		\lesssim \| g \|_{X^{k,n,m}}^2 \| f \|_{X^{k,n,m}}^2,
\end{split}\]
where we once again used the Sobolev embedding and \Cref{l:moment_interpolation}, with $m$ sufficiently large.

Thus we have
\[
	\int Q( \partial_x^{\alpha'}\partial_v^{\beta'} g, \partial_x^{\alpha''} \partial_v^{\beta''} f) \vv^{2n} \partial_x^\alpha \partial_v^\beta f \dd x \dd v
		\lesssim \|g\|_{X^{k,n,m}} \|f\|_{X^{k,n,m}}^2.
\]
which finishes the proof of case one.

\medskip

{\bf Case two: $|\alpha'| + |\beta'| = k$; i.e.\ $\alpha' = \alpha$ and $\beta' = \beta$.}
Using Cauchy-Schwarz just as we did above, it is enough to estimate the $L^{2,n}$-norm of the $Q$-term.  To this end, using \Cref{p:Q_estimates}.(i) and (ii), we find
\[
	\|Q(\partial_x^{\alpha} \partial_v^\beta g,f)\|_{L^{2,n}}
		\lesssim \|\partial_x^{\alpha} \partial_v^\beta g\|_{L^{2,n}} \big( \|f\|_{L^{\infty,m}} +  \|\vv^{n + 7/2 + \e + (\gamma+2s)_+} D_v f\|_{L^\infty_x C^\theta_v} \big),
\]
where $\theta \in ((2s-1)_+, 1)$.  Here we used that $m > n + 3/2 + \gamma + \e$.   In order to bound the $C^\theta$ norm of $D_v f$, we use \Cref{l:moment_interpolation} and the Sobolev embedding theorem.  Indeed, choosing $k' = 4 + \theta + \e$, where $\e$ is adjusted so that $k' < 5 \leq k$, we get
\[
	\|f\|_{H^{k', n + 7/2 + \e + \gamma}} \lesssim \|f\|_{X^{k,n,m}},
\]
as long as $m$ is chosen sufficiently large. 
Then, applying the Sobolev embedding theorem we  find
\[
	\|\vv^{n + 7/2 + \e + \gamma} D_v f\|_{L^\infty_x C^\theta_v}
		\leq \|f\|_{H^{k',n + 7/2 + \e + \gamma}}.
\]
Thus, we have
\[
	\int Q( \partial_x^{\alpha'}\partial_v^{\beta'} g, \partial_x^{\alpha''} \partial_v^{\beta''} f) \vv^{2n} \partial_x^\alpha \partial_v^\beta f \dd x \dd v
		\lesssim \|f\|_{X^{k,n,m}}^2 \|g\|_{X^{k,n,m}},
\]
which finishes the proof of case two.

\medskip

{\bf Case three: $|\alpha''| + |\beta''| = k$, i.e.\ $\alpha'' = \alpha$ and $\beta'' = \beta$.}
This case is handled directly by \Cref{p:Q_estimates}.(iii), which immediately yields
\[
	\begin{split}
	\int Q(g,\partial^\alpha_x \partial^\beta_v f) \vv^{2n} \partial_x^\alpha \partial_v^\beta f \dd v \dd x
		&\lesssim \int \| \partial_x^\alpha \partial_v^\beta f (x) \|_{L^{2,n}_v}^2 \| g(x) \|_{L^{\infty,m}_v} \dd x\\
		&\leq \| f \|_{H^{k,n}_{x,v}}^2 \| g \|_{L^{\infty,m}} \leq \| f \|_{X^{k,n,m}}^2 \| g \|_{X^{k,n,m}},
	\end{split}
\]
finishing the proof of case three.

\medskip

{\bf Case four: $|\alpha''| + |\beta''| = k-1$.}
Let $\partial$ be the derivative falling on $g$ instead of $\partial_x^{\alpha''} \partial_v^{\beta''} f$.  Then
$	\partial_x^{\alpha} \partial_v^{\beta} f
		= \partial \partial_x^{\alpha''} \partial_v^{\beta''} f$
and so we may apply \Cref{p:Q_estimates}.(iv).  Fixing any $\alpha \in ((2s-1)_+,s)$, we obtain
\[\begin{split}
	&\int Q(\partial g, \partial_x^{\alpha''} \partial_v^{\beta''} f) \vv^{2n} \partial \partial_x^{\alpha''} \partial_v^{\beta''} f \dd v \dd x\\
		&\lesssim \big(
				\|\partial g\|_{L^\infty_x L_v^{2,n + 3/2 + (2s -1)_+
				+ \e}} + \|\vv^{3 + (\gamma + 2s+1)_++\e}\partial g\|_{L^\infty_x C_v^\alpha} + \|\partial^2 g\|_{L^\infty_x L_v^{2,3/2 + (\gamma+2s)_+ + \e}}\big)\\
			&\qquad \cdot  \big(
				\|\partial_x^{\alpha''} \partial_v^{\beta''} f\|^2_{H^{s, n + 3 + (\gamma + 2s +1)_+ + \e}}
				+ \|\partial_x^{\alpha''} \partial_v^{\beta''} f\|^2_{H^{1,n}}\big)\\
		&\lesssim \big(
				\|g\|_{H^{5/2 + \e, n + 3/2 + (2s -1)_+ + \e}}
				+ \|g\|_{H^{4+\alpha, 3 + (\gamma + 2s+1)_+ + \e}}\big)  
			\big(
				\| f\|^2_{H^{(k - (1-s), n + 3 + (\gamma + 2s +1)_+ + \e}}
				+ \|f\|^2_{H^{k,n}}\big),
\end{split}\]
where in the second inequality we used the Sobolev embedding theorem.

Then, as long as $m$ is sufficiently large, \Cref{l:moment_interpolation} and the above inequality yield
\[
	\int Q(\partial g, \partial_x^{\alpha''} \partial_v^{\beta''} f) \vv^{2n} \partial \partial_x^{\alpha''} \partial_v^{\beta''} f \dd v \dd x\\
		\lesssim \|g\|_{X^{k,n,m}} \|f\|_{X^{k,n,m}}^2,
\]
which concludes the proof of case four, and, thus~\eqref{e:Q_key_estimate}.  Hence, the proof is complete.
\end{proof}

Having established the bounds above, we now construct a solution.  

\begin{proposition}\label{p:linear_existence}
	Fix $T>0$, $g \in Y_T^{k,n,m}$, and $f_{\rm in} \in X^{k,n,m}$.  Then there exists a unique solution $f \in Y_T^{k,n,m}$ such that
	\[
		(\partial_t + v\cdot\nabla_x) f = Q(g,f).
	\]
	and such that $f(0,\cdot,\cdot) = f_{\rm in}$.  Moreover, $f\geq 0$.
\end{proposition}
\begin{proof}
The tool that we use to construct $f$ is the method of continuity.  We do this in three steps: (i) apply Schauder estimates from the heat equation, (ii) construct a solution via the method of continuity, and (iii) use \Cref{p:a_priori} to deregularize.

{\bf Step (i).} Fix $\alpha \in (0, s(1-s)/2)$.  We now apply the classical Schauder estimates (see, e.g., \cite[Chapter 3, Theorem 6]{friedman}) to obtain, for any $f$ with the right-hand side finite,
\begin{equation}\label{e:schauder1}
	\|f\|_{C^{2,\alpha}_{\rm para}}
		\lesssim \|\cL_{\sigma,\e,\de} f\|_{C^\alpha_{\rm para}}
				+ \| \chi(\delta v) v\cdot\nabla_x f\|_{C^\alpha_{\rm para}}
				+ \| Q_{\e,\de}(g,f)\|_{C^\alpha_{\rm para}}
				+ \|f|_{t=0}\|_{C^{2,\alpha}}
\end{equation}
where $C^\alpha_{\rm para}$ (resp.\ $C^{2,\alpha}_{\rm para}$) is the classical parabolic space encoding $C^{\alpha/2}$ regularity in $t$ and $C^{\alpha}$ regularity in $x$ and $v$  (resp.\ $C^{1, \alpha/2}$ regularity in $t$ and $C^{2,\alpha}$ regularity in $x$ and $v$).  We note that the applied constant above depends only on $\e$, and that all norms in \eqref{e:schauder1} are over $[0,T]\times \T^3\times \R^3$.

We now outline how to simplify the second and third terms in~\eqref{e:schauder1}.  The third term is significantly more technical, so we only outline how to handle that term and omit discussion of the second term. From~\cite[Lemmas 6.4 and 6.9]{imbert2019smooth}\footnote{Actually, these estimates are given using the kinetic H\"older spaces so they do not apply immediately.  However, one can simply apply the estimate for each variable, thereby getting the $t$, $x$, and $v$ regularity estimates separately for the relevant H\"older spaces (in fact, this can be seen directly from the proofs in~\cite{imbert2019smooth}.} and using that we have cut-off large velocities both inside and outside the collision kernel, we have
\begin{equation}\label{e:LE1}
	\| Q_{\e,\de}(g,f)\|_{C^\alpha_{\rm para}}
	 	\lesssim \|g^\e\|_{C^{\alpha_s}_{\rm para}} \|f\|_{C^{2s + \alpha_s}_{\rm para}},
\end{equation}
where $\alpha_s = (1+2s) \alpha / (2s)$.  Using a classical interpolation lemma, we have $\|f\|_{C^{2s + \alpha_s}_{\rm para}} \leq C_\theta \|f\|_{L^\infty} + \theta \|f\|_{C^{2,\alpha}_{\rm para}}$ for any $\theta > 0$.  Note that this requires $2s + \alpha_s < 2 + \alpha$, which is true by our choice of $\alpha$ and $\alpha_s$.  Then, using the estimates in~\eqref{e:a_priori1},  we find
\[
	\|f\|_{L^\infty}
		\lesssim \|f|_{t=0}\|_{L^\infty} + \|\cL_{\sigma,\e,\de} f\|_{L^\infty}.
\]

Combining all above estimates and choosing $\theta$ sufficiently small to absorb the $\|f\|_{C^{2,\alpha}_{\rm para}}$ terms into the left-hand side, we obtain
\begin{equation}\label{e:schauder2}
	\|f\|_{C^{2,\alpha}_{\rm para}}
		\lesssim \|\cL_{\sigma,\e,\de} f\|_{C^\alpha_{\rm para}}
				+ \|f|_{t=0 }\|_{C^{2,\alpha}}.
\end{equation}

{\bf Step (ii):} Let
\[
	\cT_\sigma : C^{2,\alpha}_{\rm para}([0,T]\times \T^3 \times \R^3) \to C^\alpha_{\rm para}([0,T]\times \T^3 \times \R^3) \times C^{2,\alpha}(\T^3 \times \R^3)
\]
be defined by
\[
	\cT_{\sigma}(f)
		= \Big( \cL_{\sigma,\e,\de} f, 
			f|_{t=0}\Big).
\]
It is clear that $\cT_\sigma$ is a well-defined operator between Banach spaces and, using \Cref{e:schauder2}, there exists $C>0$, independent of $\sigma$, such that
\[
	\|f\|_{C^{2,\alpha}_{\rm para}}
		\leq C \|\cT_\sigma f\|_{C^\alpha_{\rm para} \times C^{2,\alpha}}.
\]
Finally, we notice that $\cT_0$ is onto since this corresponds to the solvability of the heat equation.  The method of continuity (see, e.g., \cite[Theorem~5.2]{gilbargtrudinger}) then implies that $\cT_\sigma$ is a bijection for all $\sigma \in [0,1]$.  Thus defining
\[
	f_{\e,\de} :=
		\cT_1^{-1}(0, \chi(\e \cdot) f^\e_{\rm in}),
\]
we have that
\[
	(\partial_t + \chi(\de v) v\cdot \nabla_x) f_{\e,\de}
		= \e \Delta_{x,v} f_{\e,\de} + Q_\e(g,f_{\e,\de})
\]
and that $f_{\e,\de}|_{t=0} = f_{\rm in}^\e$. Note that by cutting off the regularized initial data $f_{\rm in}^\eps$ with $\chi(\eps\cdot)$, we ensure the initial data will be compactly supported whenever $\eps>0$, even for $\delta = 0$.  By classical maximum principal arguments, $f_{\e,\de}\geq 0$ since $f_{\rm in} \geq 0$.

{\bf Step (iii):}  Note that $f^\e_{\rm in} \in L^{\infty,m'}$ for any $m'$.  Thus, by \Cref{p:a_priori}, we have that $f_{\e,\de}$ is bounded in $L^{\infty,m'}$ for any $m'$.  We wish to argue exactly as above using the Schauder estimates.  The main thrust of the argument is the same, the only difference being the necessity to include polynomial weights in $v$.  This causes no issues since we can use the decay of $f$; for example, instead of interpolating the $C^{\alpha}_{\rm para}$ norm of $Q_{\e,\de}(g^\e,f_{\e,\de})$ between the $C^{2,\alpha}_{\rm para}$ and $L^\infty$ norms of $f_{\e,\de}$, we interpolate between the $C^{2,\alpha}_{\rm para}$ and $L^{\infty,m'}$ norms.  If $m'$ is sufficiently large, the $L^{\infty,m'}$ norm can absorb the extra $\vv^{(\gamma+2s)_+}$ growth of $Q_{\e,\de}(g^\e,f_{\e,\de})$ (see, e.g., the estimates in \cite{imbert2019smooth}).  We conclude that
\[
	\|f_{\e,\de}\|_{C^{2,\alpha}_{\rm para}}
		\lesssim \|f^\e_{\rm in}\|_{C^{2,\alpha}} + \|f_{\e, \de}\|_{L^{\infty,m'}}
		\lesssim \|f^\e_{\rm in}\|_{C^{2,\alpha}} + \|f_{\rm in}\|_{L^{\infty,m'}},
\]
where the implied constant depends on $\e$ but not $\de$.  Thus, the Arzel\`a-Ascoli theorem implies that, up to taking a subsequence, $f_{\e,\de} \to f_\e$ locally uniformly in $C^{2,\alpha}_{\rm para}$ as $\de \to 0$.

Further, iterating the Schauder estimates, we obtain bounds on $\vv^{n'} f_\e \in L^\infty_t W^{k',\infty}$ for any $k'$ and $n'$.  Thus, $f_\e \in Y_T^{k,n,m}$; however, we need estimates on its $Y^{k,n,m}_T$-norm that are independent of $\e$.  These are provided by\footnote{Here we see the need for the two separate regularizations.  Indeed, \Cref{p:a_priori} does not yield $H^{k,n}$ bounds when $\de >0$.} \Cref{p:a_priori}.  Hence, we have that $f_\e$ is bounded in $Y_T^{k,n,m}$ independently of $\e$.   From here, it is standard to obtain a limit $f$ of $f_\e$ via compactness, so we omit the details.  This concludes the construction of $f$.  Since limits in Sobolev and H\"older spaces preserve non-negativity, we find $f\geq 0$ since $f_{\e,\de}\geq 0$.  This concludes the proof of the proposition.
\end{proof}

\subsection{Solving the full nonlinear equation: the proof of \Cref{t:main}}

We are now in a position to prove our main theorem via an iteration argument.

\begin{proof}[Proof of \Cref{t:main}]
In order to show uniqueness, we note that, due to \Cref{l:moment_interpolation}, $L^{\infty,m} \cap H^{k,n}  \subset H^{2s, 14}$ as long as $m$ is sufficiently large.  Hence, uniqueness follows from \cite[Theorem 1.1]{amuxy2011uniqueness}.

Let $f_0 = f_{\rm in}$.  For each $i \in \N$, let $f_i$ be the unique solution to
\begin{equation}\label{e:iteration}
 	(\partial_t + v\cdot\nabla_x) f_i
		= Q(f_{i-1}, f_i)
		\qquad \text{ in } (0,\infty) \times \dom
\end{equation}
with initial data $f_{\rm in}$ that is guaranteed by \Cref{p:linear_existence}.  Let
\[
	T_0 = \min\{T_1, T_2\},
\]
where $T_1$ and $T_2$ are to be determined and depend only on $\|f_{\rm in}\|_{X^{k,n,m}}$, $\gamma$, $s$, $k$, $n$, and $m$.

We first establish that the sequence $f_i$ is bounded in $Y_{T_0}^{k,n,m}$.  We do this inductively.  We claim that $\|f_i\|_{Y_{T_0}^{k,n,m}} \leq 2 \|f_{\rm in}\|_{X^{k,n,m}}$ for every $i$.  This is clearly true for the case $i=0$, hence, we assume it is true for $i$ and show that it holds for $i+1$.  From \Cref{p:a_priori}, we have
\[
	\|f_{i+1}\|_{Y^{k,n,m}_{T_0}}^2
		\leq (1+CT_0) \|f_{\rm in}\|_{X^{k,n,m}}^2 \exp\Big\{C\int_0^{T_0} \|f_{i}\|^2_{X^{k,n,m}} \dd t\Big\},
\]
where $C$ is independent of $i$.  Then, by the inductive hypothesis, we have
\[
	\|f_{i+1}\|^2_{Y^{k,n,m}_{T_0}}
		\leq (1+CT_0)\|f_{\rm in}\|_{X^{k,n,m}}^2 \exp\Big\{4C T_0 \|f_{\rm in}\|^2_{X^{k,n,m}}\Big\}.
\]
The bound on $f_{i+1}$ is now finished by defining
\[
	T_1 = \min\Big\{\frac{1}{4C \|f_{\rm in}\|_{X^{k,n,m}}^2} \log(2), 1/C\Big\}
\]
and using that $T_0 \leq T_1$.

We note that, since $f_i$ satisfies~\eqref{e:iteration} and is an element of $Y_{T_0}^{k,n,m}$, $f_i\in W^{1,\infty}_t H^{k-2s,n-2}_{x,v}$ with bounds uniform in $i$.  Differentiating \eqref{e:iteration} in time, we find that $f_i \in W^{2,\infty}_t H^{k-4s,n-4}_{x,v}$ with bounds likewise uniform in $i$.

We now obtain a limit.  For any $i$, let $F_i = (f_i, f_{i+1}) \in Y_{T_0}^{k,n,m} \times Y_{T_0}^{k,n,m}$.  Thus, taking a weak limit in $Y_{T_0}^{k,n,m}$ and a strong limit in any weaker space (e.g.\ $C^{1,\alpha}_t H^{1,n-4}_{x,v} \cap C^\alpha_tH^{2,n-2}_{x,v}$), we find that there must exist a pair $F = (\bar f_1, \bar f_2)$ such that
\[
	(\partial_t + v\cdot \nabla_x)\bar f_2 = Q(\bar f_1, \bar f_2).
\]

Hence, our proof of existence is concluded after showing that $\bar f_1 = \bar f_2$.  Let $w_i = f_{i+1} - f_i$.  If we establish that $\|w_i\|_{L^{2,n}} \to 0$ as $i\to\infty$, it follows that $\bar f_1 = \bar f_2$.  Notice that
\[
	(\partial_t + v\cdot\nabla_x) w_i
		= Q(f_i, w_i) + Q(w_{i-1}, f_i).
\]
Multiplying this by $\vv^{2n} w_i$ and integrating, we find
\[
	\frac{1}{2} \ddt \int \vv^{2n} |w_i|^2 \dd v \dd x
		= \int \vv^{2n} Q(f_i,w_i) w_i \dd v \dd x + \int \vv^{2n} Q(w_{i-1}, f_i) w_i \dd v \dd x.
\]
We estimate the first term using \Cref{p:Q_estimates}.(iii) and the second term we estimate using \Cref{p:Q_estimates}.(i)-(ii) in order to get
\[\begin{split}
	\frac{1}{2} \ddt \int \vv^{2n} |w_i|^2 \dd v \dd x
		&\lesssim \|f_i\|_{L^{\infty,m}} \|w_i\|_{L^{2,n}}^2 + \|f_i\|_{X^{k,n,m}} \|w_{i-1}\|_{L^{2,n}} \|w_i\|_{L^{2,n}}\\
		&\leq \left( \|f_i\|_{L^{\infty,m}}  + \frac 1 2 \| f_i\|_{X^{k,n,m}}^2 \right) \|w_i\|_{L^{2,n}}^2 + \frac 1 2 \| w_{i-1} \|_{L^{2,n}}^2.
\end{split}\]
Using that $\| f_i \|_{L^{\infty,m}} \leq C (1+ \| f_i \|_{X^{k,n,m}})$ we then have
\[
	\frac 1 2 \ddt \int \vv^{2n} |w_i|^2 \dd v \dd x \lesssim C (1+\| f_{\rm in} \|_{X^{k,n,m}})^2 \|w_i \|_{L^{2,n}}^2
		+ \frac 1 2 \| w_{i-1} \|_{L^{2,n}}^2.
\]
Integrating this in time, using that $w_i(0,\cdot,\cdot) \equiv 0$, and taking a supremum in $t$ yields
\[
	\|w_i\|_{L_t^\infty L^{2,n}}^2
		\leq 2 C T_0 (1+\|f_{\rm in}\|_{X^{k,n,m}})^2 \| w_i \|_{L_t^\infty L^{2,n}}^2
			+  \frac{T_0}{2} \| w_{i-1} \|_{L_t^\infty L^{2,n}}^2.
\]
We now choose $T_2 = \min\{\frac 1 2 ,(4 C (1+\|f_{\rm in}\|_{X^{k,n,m}})^2)^{-1}\}$.  Since $T_0 \leq T_2$, we obtain
\[
	\|w_i\|_{L_t^\infty L^{2,n}}^2
		\leq \frac{1}{2} \|w_{i-1}\|_{L_t^\infty L^{2,n}}^2,
\]
which implies that $w_i \to 0$ in $L^{2,n}$.  As noted above, this completes the proof of existence.

Finally, we note that $f\in C([0,T]; H_{x,v}^{k,n})$ by standard energy methods along with the bounds above (see, e.g., the proof of  \cite[Theorem 4.1]{amuxy2010regularizing}) and that $f\geq 0$ since $f_i \geq 0$ for all $i$.
\end{proof}

\section{Estimates on $Q$: the proof of \Cref{p:Q_estimates}} \label{s:Q}


\subsection{The estimate on $\Qs$: the proof of \Cref{p:Q_estimates}.(i)}\label{s:Qs-proof}
\begin{proof}[Proof of \Cref{p:Q_estimates}.(i)]

Recall the formula for $\Qs$ from~\Cref{l:Qs}.  We begin by using an annular decomposition of $\Qs(g,f)(v)$ for any fixed $v \in \R^3$.    Let $A_k(v) = B_{2^k |v|} (v) \setminus B_{2^{k-1}|v|}(v)$, and, for convenience, let $\mu = n + 7/2 + \gamma + \e$ and $\eta = 3/2 + \e + (\gamma+2s)_+$. 
We write
\[\begin{split}
	\Qs(g,f)
		= \sum_{k\in \Z} \int_{A_k(v)} K_g(v,v') (f(v') - f(v)) \dd v'.
\end{split}\]

\medskip

{\bf Step One: estimating the sum for any $k\leq -1$.}  Taylor expanding $f$ at $v$, we find
\[
	f(v') - f(v) = \left((Df)(\xi_{v,v'}) - (Df)(v)\right)\cdot (v'-v) + (Df)(v)\cdot (v'-v)
\]
for some $\xi_{v,v'}$ on the line segment connecting $v$ and $v'$.  Thus,
\[\begin{split}
	&\int_{A_k(v)} K_g(v,v') (f(v') - f(v)) \dd v' \\
		&= \int_{A_k(v)}  K_g(v,v') ((Df)(\xi_{v,v'}) - (Df)(v)) \cdot (v-v') \dd v'
			+ \int_{A_k(v)}  K_g(v,v') (Df)(v)(v-v') \dd v'.
\end{split}\]
It is easy to see that $K_g(v,v') = K_g(v, v - (v'-v))$.  Hence, the last integral above vanishes by symmetry.  The remaining integral can be estimated using the $\|\langle\cdot\rangle^{\mu} Df\|_{C^\alpha}$ norm of $f$.  After this, we use~\Cref{l:kernel_bounds}, which yields
\begin{equation}\label{e:c001}
\begin{split}
	\Big|\int_{A_k(v)} K_g(v,v') &(f(v') - f(v)) \dd v'\Big|
		\leq \langle v \rangle^{-\mu} \|\langle \cdot \rangle^{\mu} Df\|_{C^\alpha} \int_{A_k(v)}  K_g(v,v') |v-v'|^{1+\alpha} \dd v'\\
		&\lesssim \vv^{-\mu}(2^k|v|)^{1+\alpha-2s} \|\langle \cdot \rangle^{\mu} Df\|_{C^\alpha} \int_{\R^3}  |g(v')| |v-v'|^{\gamma + 2s} \dd v'.
\end{split}
\end{equation}
Recall that $1+\alpha - 2s > 0$ by assumption.  We now apply Cauchy-Schwarz and a straightforward estimate of the convolution of algebraic functions to find
\begin{equation}\label{e:Q_step12}
\begin{split}
	&\left|\int_{A_k(v)} K_g(v,v') (f(v') - f(v)) \dd v'\right|\\
		&\quad \lesssim \vv^{-\mu+1+\alpha-2s}(2^k)^{1+\alpha-2s} \|\langle \cdot \rangle^{\mu} Df\|_{C^\alpha} \|g\|_{L^{2,\eta}} \left(\int \langle v'\rangle^{-2\eta}|v-v'|^{2(\gamma + 2s)} \dd v'\right)^{1/2}\\
		&\quad \lesssim \vv^{-\mu + 1 + \alpha - 2s }(2^k)^{1+\alpha-2s} \|\langle \cdot \rangle^{\mu} Df\|_{C^\alpha} \|g\|_{L^{2,\eta}} \vv^{\gamma + 2s}.
\end{split}
\end{equation} 
Above we used that $\gamma + 2s > -3/2$ so that $|v-v'|^{2(\gamma+2s)}$ is integrable near $v$ and that $\eta > 3/2 + (\gamma + 2s)_+$.

Summing over all $k\leq -1$, we obtain
\[\begin{split}
	\sum_{k\leq 1} \left|\int_{A_k(v)} K_g(v,v') (f(v') - f(v)) \dd v'\right|
		&\lesssim \sum_{k\leq-1} \vv^{-\mu+1+\alpha + \gamma}(2^k)^{1+\alpha-2s} \|\langle \cdot \rangle^{\mu} Df\|_{C^\alpha} \|g\|_{L^{2,\eta}}\\
		&\lesssim \vv^{ - \mu+ 1 + \alpha + \gamma}\|\langle \cdot \rangle^{\mu} Df\|_{C^\alpha} \|g\|_{L^{2,\eta}}
\end{split}\]
where we used once more that $1+\alpha>2s$.  Since $2(n + 1 + \alpha + \gamma - \mu) < -3$,
\begin{equation}\label{e:Q_step1}
	\int \vv^{2n}\left|\sum_{k\leq -1} \int_{A_k(v)} K_g(v,v') (f(v') - f(v)) \dd v'\right|^2 \dd v
		\lesssim \|\langle \cdot \rangle^{\mu} Df\|_{C^\alpha}^2 \|g\|_{L^{2,\eta}}^2.
\end{equation}
This is the desired estimate since $\mu \leq n + \eta + 2$, so Step One is complete.

\medskip

{\bf Step Two: estimating the sum for $k\geq 0$ when $|v'| \geq \vv/2$.} 
Fix any $k\geq 0$.  Using \Cref{l:kernel_bounds}, we find
\[\begin{split}
	&\left|\int_{A_k \setminus B_{\vv/2}} K_g(v,v') (f(v') - f(v)) \dd v'\right|
		\lesssim  \vv^{-m} \|f\|_{L^{\infty,m}} \int_{A_k\setminus B_{\vv/2}} |K_g(v,v')| \dd v' \\
		&\qquad \leq \vv^{-m} \|f\|_{L^{\infty,m}} \int_{A_k} |K_g(v,v')| \dd v' 
		\lesssim \vv^{-m} \|f\|_{L^{\infty,m}}  (2^k|v|)^{-2s}\int_{\R^3} |g(v')| |v-v'|^{\gamma + 2s} \dd v'.
\end{split}\]
We now apply Cauchy-Schwarz exactly as we did in~\eqref{e:Q_step12} to obtain
\[\begin{split}
	&\left|\int_{A_k \setminus B_{\vv/2}} K_g(v,v') (f(v') - f(v)) \dd v'\right|
		\lesssim \vv^{-m + \gamma + 2s - \e} \|f\|_{L^{\infty,m}}  (2^k)^{-2s}
		|v|^{-2s} \|g\|_{L^{2,\eta}}.
\end{split}\]

Unfortunately, $|v|^{-4s}$ is not integrable near $v= 0$ in general.  However, we notice the following.  Let $k_v = (-1 - \log|v|)_+$, and, if $0 \leq k < k_v$, then by a short calculation, $A_k \setminus B_{\vv/2} = \emptyset$.  Thus,
\begin{equation}\label{e:Q_step2}
\begin{split}
	\int \vv^{2n} &\left|\sum_{k \geq 0} \int_{A_k \setminus B_{\vv/2}} K_g(v,v') (f(v') - f(v)) \dd v'\right|^2 \dd v\\
		&= \int \vv^{2n} \left|\sum_{k \geq k_v} \int_{A_k \setminus B_{\vv/2}} K_g(v,v') (f(v') - f(v)) \dd v'\right|^2 \dd v\\
		&\lesssim \|f\|_{L^{\infty,m}}^2 \|g\|_{L^{2,\eta}}^2 \int_{\R^3} \vv^{2(n -m + \gamma + 2s - \e)} |v|^{-4s} \left| \sum_{k\geq k_v} 2^{-2sk}\right|^2 \dd v\\
		&\lesssim \|f\|_{L^{\infty,m}}^2 \|g\|_{L^{2,\eta}}^2 \int_{\R^3} \vv^{2(n -m + \gamma + 2s - \e)} \vv^{-4s}  \dd v
		\lesssim \|f\|_{L^{\infty,m}}^2 \|g\|_{L^{2,3/2 + \e}}^2.
\end{split}
\end{equation}
In the second-to-last inequality, we used that
\[
	|v|^{-4s}\left|\sum_{k\geq k_v} 2^{-2sk}\right|^2
		\lesssim |v|^{-4s} \left|2^{-2sk_v}\right|^2
		\lesssim |v|^{-4s}\min\{1,|v|^{4s}\}
		\lesssim \vv^{-4s},
\]
and in the last inequality, we used that $m$ is sufficiently large so that $2(n - m + \gamma - \e) < - 3$.  This concludes Step Two.


\medskip

{\bf Step Three: estimating the sum for $k\geq 0$ when $|v| \leq 10$ and $|v'| \leq \vv/2$.} 
When $|v|\leq 10$, this estimate is straightforward because we do not have the issue of weights as the integrals are on compact sets.  In fact, the estimate can be done exactly as in Step One.  Hence, we omit the proof and assert that
\begin{equation}\label{e:Q_step3}
\begin{split}
	\int_{B_{10}} \vv^{2n} &\left|\sum_{k \geq 0} \int_{A_k \cap B_{\vv/2}} K_g(v,v') (f(v') - f(v)) \dd v'\right|^2 \dd v\\
		&\lesssim \int_{B_{10}} \vv^{2n} \left|\int_{B_{\vv/2}} K_g(v,v') (f(v') - f(v)) \dd v'\right|^2 \dd v
		\lesssim \|Df\|_{C^\alpha}^2 \|g\|_{L^{2,3/2+\e}}^2.
\end{split}
\end{equation}
This concludes Step Three.

\medskip

{\bf Step Four: estimating the sum for $k\geq 0$ when $|v| \geq 10$ and $|v'| \leq \vv/2$.} 
We now consider the final portion of $\Qs(g,f)$.  Fix any $v$ such that $|v|\geq 10$ and any $k \geq 0$.  First, by the triangle inequality, we see that
\begin{equation}\label{e:Q_step41}
\begin{split}
	\Big|\sum_{k\geq 0} &\int_{A_k\cap B_{\vv/2}} K_g(v,v') (f(v') - f(v)) \dd v'\Big|\\
		&\lesssim \int_{B_{\vv/2} } K_{|g|}(v,v') |f(v')| \dd v' + \int_{B_{\vv/2}} K_{|g|}(v,v') |f(v)| \dd v'.
\end{split}
\end{equation}
The second term on the right hand side is easy to bound.  Notice that $B_{\vv/2} \subset (B_{2\vv}(v)\setminus B_{\vv/4}(v))$ since $|v|\geq 10$.  Then, using \Cref{l:kernel_bounds}, we find
\[\begin{split}
	\int_{B_{\vv/2}} K_{|g|}(v,v') \dd v'
		&\lesssim  \int_{B_{\vv/2}} K_{|g|}(v,v') \dd v'
		\lesssim \int_{B_{2\vv}(v) \setminus B_{\vv/4}(v)} K_{|g|}(v,v') \dd v'\\
		&\lesssim \vv^{ -2s}  \int |g(v')||v-v'|^{\gamma + 2s} \dd v'
		\lesssim \vv^{\gamma}  \|g\|_{L^{2,\eta}}.
\end{split}\]
Thus, using that $\eta \leq n$,
\begin{equation}\label{e:Q_step42}
	\int_{B_{10}^c} \vv^{2n} \Big(\int_{B_{\vv/2}} K_{|g|}(v,v') |f(v)| \dd v' \Big)^2 \dd v
		\lesssim \|g\|_{L^{2,n}}^2 \|f\|_{L^{2, n}}^2.
\end{equation}

On the other hand, the first term in~\eqref{e:Q_step41} requires a bit more work.  Fix $\e \in (0,|\gamma|)$.  Applying Cauchy-Schwarz twice and using the definition of $K$, we find
\[\begin{split}
	\Big(\int_{B_{\vv/2}} &K_{|g|}(v,v') |f(v')| \dd v'\Big)^2
		\lesssim \|f\|_{L_v^{2,3/2 + \e}}^2 \int_{B_{\vv/2}} \vvp^{-(3 + 2\e)} K_{|g|}(v,v')^2  \dd v'\\
		&= \|f\|_{L_v^{2,3/2 + \e}}^2 \int_{B_{\vv/2}}\frac{\vvp^{-(3 + 2\e)}\left(\int_{(v-v')^\perp} |g(v+w)| |w|^{\gamma + 2s + 1} \ww^{1+\e}  \ww^{-1-\e}  \dd w\right)^2}{|v-v'|^{2(3+2s)}} \dd v'\\
		&\lesssim \|f\|_{L_v^{2,3/2 + \e}}^2 \int_{B_{\vv/2}}\frac{\vvp^{-(3 + 2\e)}\int_{(v-v')^\perp} g(v+w)^2 \langle w\rangle^{2(\gamma + 2s + 2 + \e)} \dd w}{|v-v'|^{2(3+2s)}} \dd v'.
\end{split}
\]
Let $H(w):= \langle w\rangle^{2n+1} g(w)^2$.  Notice that, since $|v|>10$, then $|v| \geq 10 \vv / 11$, which implies that $|v-v'| \approx \vv$ (recall that $v' \in B_{\vv/2}$).  In addition, since $w \perp (v-v')$, we see that
\begin{equation}\label{e:andrei1}
\begin{split}
	|v+w|^2
		&= |v|^2 + 2 (v-v')\cdot w + 2v' \cdot w + |w|^2
		\geq |v|^2 - \frac{3}{2}|v'|^2 - \frac{2}{3}|w|^2 + |w|^2\\
		&\geq |v|^2 - \frac{3}{4}\vv^2 + \frac{1}{3}|w|^2
		\geq |v|^2 - \frac{3}{4}\frac{121}{100}|v|^2 + \frac{1}{3}|w|^2
		\gtrsim |v|^2 + |w|^2,
\end{split}
\end{equation}
and, hence, $|v+w| \approx |v| + |w|$.  Using these observations, we find
\[\begin{split}
	\int_{B_{\vv/2}}&\frac{\vvp^{-(3 + 2\e)}\int_{(v-v')^\perp} g(v+w)^2 \langle w\rangle^{2(\gamma + 2s + 2+\e)} \dd w}{|v-v'|^{2(3+2s)}} \dd v'\\
		&\lesssim \int_{B_{\vv/2}}\frac{\vvp^{-(3 + 2\e)}}{\vv^{2(3+2s)}} \int_{(v-v')^\perp} \frac{H(v+w)}{\vv^{2n+1} + \langle w\rangle^{2n+1}} \langle w\rangle^{2(\gamma + 2s + 2 + \e)} \dd w \dd v' \dd v.
\end{split}\]

Next, we notice that $\sup_w \langle w\rangle^{2(\gamma+2s+2 + \e)}(\vv^{2n+1} + \langle w\rangle^{2n+1})^{-1} \lesssim \vv^{2(\gamma + 2s + \e) + 3 - 2n}$.  Using this and then spherical coordinates in the $v$ variable ($\rho = |v|$, $z\in \partial B_\rho$) yields
\[\begin{split}
	\int_{B_{10}^c} &\vv^{2n} \Big(\int_{B_{\vv/2}} K_{|g|}(v,v') |f(v')| \dd v'\Big)^2 \dd v\\
		&\lesssim \|f\|_{L_v^{2,3/2 + \e}}^2\int_{B_{10}^c} \vv^{2(\gamma+ \e)-3} \int_{B_{\vv/2}}\vvp^{-(3 + 2\e)} \int_{(v-v')^\perp} H(v+w) \dd w \dd v' \dd v\\
		&\lesssim \|f\|_{L_v^{2,3/2 + \e}}^2 \int_{10}^\infty \rho^{2(\gamma+\e)-3} \int_{B_{\rho/2}}   \vvp^{-(3 + 2\e)} \left(\int_{\partial B_\rho}  \int_{(v-z)^\perp} H(z+w) \dd w \dd z\right) \dd v' \dd \rho.
%
%
%
%
%
%
%
\end{split}
\]

At this point, we apply Lemma \ref{l:henderson-problem} to the $w$ and $z$ integrals to obtain
\[\begin{split}
	\int_{B_{10}^c} &\vv^{2n} \Big(\int_{B_{\vv/2}} K_{|g|}(v,v') |f(v')| \dd v'\Big)^2 \dd v\\
		&\lesssim \|f\|_{L_v^{2,3/2 + \e}}^2\int_{10}^\infty \int_{B_{\rho/2}}  \rho^{2(\gamma + \e) -3}   \vvp^{-(3 + 2\e)}  \left(\rho^2 \int_{\R^3\setminus B_{\rho/2}} \frac{H(w)}{|w|} \dd w\right) \dd v' \dd \rho \\
		&\lesssim \|f\|_{L_v^{2,3/2 + \e}}^2\int_{10}^\infty \int_{B_{\rho/2}}  \rho^{2(\gamma+\e)-1}   \vvp^{-(3 + 2\e)}  \left(\int_{\R^3} \langle w\rangle^{2n} g(w)^2 \dd w\right) \dd v' \dd \rho\\
		&= \|f\|_{L_v^{2,3/2 + \e}}^2 \|g\|_{L^{2,n}}^2 \int_{10}^\infty \int_{B_{\rho/2}}  \rho^{2(\gamma+\e)-1}   \vvp^{-(3 + 2\e)} \dd v' \dd \rho 
		\lesssim \|f\|_{L_v^{2,3/2 + \e}}^2 \|g\|_{L^{2,n}}^2.
\end{split}
\]
The last inequality follows from the fact that $\e \in(0,|\gamma|)$ so that $\gamma + \e < 0$.

Putting this together with~\eqref{e:Q_step42} yields, in view of~\eqref{e:Q_step41},
\begin{equation}\label{e:Q_step4}
		\int_{B_{10}^c}\Big|\sum_{k\geq 0} \int_{A_k\cap B_{\vv/2}} K_g(v,v') (f(v') - f(v)) \dd v'\Big|^2 \dd v
			\lesssim \|f\|_{L_v^{2,3/2 + \e}}^2 \|g\|_{L^{2,n}}^2	.
\end{equation}

The proof is finished after compiling~\eqref{e:Q_step1}, \eqref{e:Q_step2}, \eqref{e:Q_step3}, and \eqref{e:Q_step4}.
\end{proof}

\subsection{An estimate on $\int \vv^{2n} Q(g,f)f \dd v$: the proof of \Cref{p:Q_estimates}.(iii)}

\begin{proof}[Proof of \Cref{p:Q_estimates}.(iii)] 
This proof is very similar to \cite[Section 4, ``Estimate of $A_1$'']{amuxy2011bounded}; though, their proof does not work out of the box as they use the Gaussian decay of $f$ assumed in that paper.  As such, we sketch steps which are exactly the same; the interested reader can consult~\cite{amuxy2011bounded}.

Let $F(v) := \vv^n f(v)$, then we have
\begin{equation}
\begin{split}
\int \vv^{2n} Q(g,f) f \dd v &= \int F Q(g, F) dv + \int F \left( \vv^n Q(g,f) - Q(g,\vv^n f) \right) \dd v
	=: I_1 + I_2.
\end{split}
\end{equation}
The estimate on $I_2$ will include a term that we cannot bound using the $Y^{k,n,m}_T$-norm of $f$ or $g$. Instead,
the $I_1$ term, being symmetric, will provide a corresponding negative term of the same order which will close the estimate
on $I_2$.   For ease of notation, we replace $v_*$ with $w$. A useful tool is the so-called pre/post collisional change of variables:
\begin{equation}\label{e:pre-post}
\begin{split}
v &\rightarrow \frac{v+w}{2} + \frac{|v-w|}{2}\sigma = v', \ \
w \rightarrow \frac{v+w}{2} - \frac{|v-w|}{2}\sigma = w', \ \
\sigma \rightarrow \frac{v-w}{|v-w|} =: \sigma'.
\end{split}
\end{equation}
This change has unit Jacobian, and reflects the ``micro-reversibility'' of the collision operator. Moreover, it leaves
the quantities $\cos\theta = \sigma \cdot \sigma'$ and $|v-w|$ unchanged.

\medskip

{\bf Step One: estimating $I_1$ with a coercive contribution.} Following~\cite{amuxy2011bounded}, we have that
\begin{equation}
\begin{split}
I_1 &
	= -\frac 1 2 D + \int Q(g,F^2) \dd v,
\end{split}
\end{equation}
where
\begin{equation}\label{e:prop_3_1_D}
D = \int_{\R^6 \times \bbS} \left( F(v') - F(v) \right)^2 g(w) B(|v-w|,\cos\theta)\dd \sigma \dd w \dd v.
\end{equation}

Recall that $Q(g,F^2) = \Qs(g,F^2) + \Qns(g,F^2)$.  To bound the $Q_{\rm s}$ term, we use a change of variables and Lemma \ref{l:kernel_bounds} to get
\begin{equation}
\begin{split}
\int Q_{\text{s}}(g,F^2) \dd v &= \frac 1 2 \int F(v)^2  \int \left( K_g(v',v) - K_g(v,v') \right) \dd v' \dd v \\
&\lesssim \| F \|_{L^2_v}^2 \sup_{v} \left( \int |g(z)| |v-z|^\gamma \dd z \right)
\leq \| f \|_{L^{2,n}}^2 \| g \|_{L^{\infty,m}}.
\end{split}
\end{equation}

The nonsingular term is handled by easily using \Cref{l:Qns}:
\begin{equation}
\int Q_{\text{ns}}(g,F^2) \dd v \leq C_S \| F \|_{L^2}^2 \sup_{v} \left( \int g(z) |v-z|^\gamma dz \right)
\lesssim \| f \|_{L^{2,n}}^2 \| g \|_{L^{\infty,m}}.
\end{equation}

Thus,
\begin{equation}\label{e:prop_3_1_I1}
I_1 + \frac{1}{2} D \lesssim \| f \|_{L^{2,n}}^2 \| g \|_{L^{\infty,m}}.
\end{equation}

{\bf Step Two: estimating $I_2$.} Using \eqref{e:Q}, we have
\begin{equation}
\begin{split}
&\int F\left( \vv^n Q(g,f) - Q(g,\vv^n f) \right) \dd v\\
&\quad\quad\quad\quad = \int_{\R^6\times \bbS} \left( \vv^n - \langle v' \rangle^n \right) f(v') F(v) g(w') B(|v-w|,\cos\theta) \dd \sigma \dd w \dd v.
\end{split}
\end{equation}
Since $v'$ is a dependent variable, it would be very difficult to estimate integrals of $f(v')$, especially with weights. Instead, we
employ the pre/post-collisional change of variables \eqref{e:pre-post}.  
Thus,
\begin{equation}
\begin{split}
I_2 &= \int_{\R^6\times \bbS} \left( \langle v' \rangle^n - \vv^n \right) f(v) F(v') g(w) B \dd\sigma \dd w \dd v \\
&= \int_{\R^6} f(v) F(v) g(w) \int_{\bbS} \left( \langle v' \rangle^n - \vv^n \right) B \dd\sigma \dd w \dd v \\
&\quad + \int_{\R^6\times \bbS} \left( \langle v' \rangle^n - \vv^n \right) f(v) \left( F(v') - F(v) \right)
g(w) B \dd\sigma \dd w \dd v
	=: I_{21} + I_{22}.
\end{split}
\end{equation}

For $I_{21}$, we Taylor expand the difference in weights to obtain
\begin{equation}\label{e:prop_3_1_Taylor}
\langle v' \rangle^n - \vv^n = n \vv^{n-2} v \cdot (v' - v) + \frac{n(n-2)}{2} \langle \tilde{v} \rangle^{n-4} \left( \tilde{v} \cdot (v'-v) \right)^2,
\end{equation}
where $\tilde{v} = \tau v' + (1-\tau) v$ for some $\tau \in (0,1)$. By \eqref{e:Q'}, we also note that
\begin{equation}\label{e:prop_3_1_v'v}
v'-v = \frac{|v-w|}{2}(\sigma-\sigma'\cos\theta) + \frac{|v-w|}{2}(\cos\theta-1)\sigma',
\end{equation}
so in particular
$
|v'-v|^2 = \frac 1 2 |v-w|^2 (1-\cos\theta).
$ 
Using the pre/post-collisional change of variables makes the integration in $\sigma$ much simpler
in $I_{21}$. Indeed, using \eqref{e:prop_3_1_Taylor} and \eqref{e:prop_3_1_v'v}, we see that
\begin{equation}\label{e:prop_3_1_angles}
\begin{split}
&\int_{\bbS} \left( \langle v' \rangle^n-\vv^n \right) B \dd\sigma  
 	=\frac{n}{2} \vv^{n-2} |v-w|v \cdot \int_{\bbS} (\sigma - (\sigma \cdot \sigma') \sigma') B \dd\sigma\\
 	& + \frac n 2 \vv^{n-2} |v-w| v \cdot \sigma' \int_{\bbS} (\cos\theta - 1) B \dd\sigma
	 + \frac{n(n-2)}{2} \int_{\bbS} \langle \tilde{v} \rangle^{n-4} \left( \tilde{v} \cdot (v'-v) \right)^2
B \dd\sigma,
\end{split}
\end{equation}
remembering that $\sigma'$ is independent of $\sigma$, but $\tilde{v}$ is not.

The first term of \eqref{e:prop_3_1_angles} is zero; the change of variables $\sigma \rightarrow -\sigma + 2 (\sigma \cdot \sigma') \sigma'$
leaves the collision kernel $B$ unchanged, but changes the sign of the integrand. Using \eqref{e:B}, the second term of \eqref{e:prop_3_1_angles}
is bounded by
\begin{equation}
\vv^{n-1} |v-w|^{1+\gamma} \int_{\bbS} (\cos\theta-1)|\theta|^{-2-2s} \dd\sigma \lesssim \vv^{n-1}|v-w|^{1+\gamma}.
\end{equation}
For the third term, we must estimate the size of $\tilde{v}$ in terms of $v$ and $w$. Since we have conservation of energy ($|v|^2 + |w|^2
= |v'|^2 + |w'|^2$), we in particular have that $\langle v' \rangle^2 \leq \vv^2 + \langle w \rangle^2$, which implies that
\begin{equation}
\langle \tilde v \rangle^{n-2} = \langle \tau v' + (1-\tau) v \rangle^{n-2} \lesssim \langle v' \rangle^{n-2} + \vv^{n-2}
\lesssim \langle w \rangle^{n-2} + \vv^{n-2}.
\end{equation}
Using the above and that $|v'-v|^2 = \frac12 |v-w|^2 (1-\cos(\theta))$, the third term is bounded above by
\begin{equation}\label{e:prop_3_1_Taylor_end}
\begin{split}
	&\left( \langle w \rangle^{n-2} + \vv^{n-2} \right) \frac{|v-w|^{2+\gamma}}{2} \int_{\bbS} (1-\cos\theta) |\theta|^{-2-2s} \dd\sigma 
		\lesssim \left \langle w \rangle^{n-2} + \vv^{n-2} \right) |v-w|^{2+\gamma}.
\end{split}
\end{equation}
Then substituting the estimates for \eqref{e:prop_3_1_angles} into the expression for $I_{21}$
(and remembering that $m > n+3+\gamma$) yields
\begin{equation}
\begin{split}
|I_{21}| 
	&\lesssim \int F(v) f(v)  \int g(w) \left(\vv^{n-1} |v-w|^{1+\gamma} + \langle w \rangle^{n-2} |v-w|^{2+\gamma}  + \vv^{n-2} |v-w|^{2+\gamma}\right)\dd w \dd v\\
&\lesssim \int F(v) f(v)\| g \|_{L^{\infty,m}} \left( \vv^{n-1} \vv^{1+\gamma} + \vv^{2+\gamma}
+ \vv^{n-2} \vv^{2+\gamma} \right) \dd v
\lesssim \| f \|_{L^{2,n}}^2 \| g \|_{L^{\infty,m}}.
\end{split}
\end{equation}

The remaining term $I_{22}$ is outwardly more complicated because it involves $F(v')$. However, this term can be absorbed
into the negative contribution from $I_1$. Specifically, we use \eqref{e:prop_3_1_D} and the Cauchy-Schwartz and Young's inequalities to obtain
\begin{equation}
\begin{split}
|I_{22}| 
&\leq \frac 1 2 D + \frac 1 2 \int\int f(v)^2 g(w) \int_{\bbS} \left(\langle v' \rangle^n - \vv^n \right)^2 B \dd\sigma \dd w \dd v.
\end{split}
\end{equation}
The second term above can be bounded in the same way as $I_{21}$. Indeed, one way to see this is to write
$(\vv^n - \langle v' \rangle^n)^2 = 2\vv^n (\vv^n - \langle v' \rangle^n) - (\vv^{2n} - \langle v' \rangle^{2n})$. Therefore,
\begin{equation}
\begin{split}
I_{22} &\leq \frac 1 2 D + \int\int f(v)^2 g(w) \vv^n \int_{\bbS} \left( \vv^n - \langle v' \rangle^n \right) B \dd \sigma \dd w \dd v\\
&\quad + \int\int f(v)^2 g(w) \int_{\bbS} \left( \langle v' \rangle^{2n} - \vv^{2n} \right) B \dd\sigma \dd w \dd v
=: \frac 1 2 D - I_{21} + \bar{I}_{22}.
\end{split}
\end{equation}
Clearly the second term can be bounded exactly as above. The third term is bounded similarly, yielding the estimates \eqref{e:prop_3_1_Taylor} --
\eqref{e:prop_3_1_Taylor_end} with $2n$ in place of $n$. Then, since $m > 2n + 3 + \gamma$, we have
\begin{equation}
\begin{split}
|\bar{I}_{22}| \lesssim \int f(v)^2 \left( \vv^{2n-1} \| g \|_{L^{\infty,m}} \vv^{1+\gamma}
+ \| g \|_{L^{\infty,m}} \vv^{2+\gamma} + \vv^{2n-2} \| g \|_{L^{\infty,m}} \vv^{2+\gamma} \right) \dd v.
\end{split}
\end{equation}
This completes the proof of Proposition \ref{p:Q_estimates}.(iii).

\end{proof}

\subsection{An estimate on $\int \vv^{2n} Q(g,f)\partial f \dd v$: the proof of \Cref{p:Q_estimates}.(iv)}

\begin{proof}[Proof of \Cref{p:Q_estimates}.(iv)]

The proof is simpler when $\partial = \partial_{x_i}$, so we omit this proof and show the more involved case when $\partial = \partial_{v_i}$.  As a convenient consequence, the $x$ variable plays no role in our proof, so we omit it from all notation.

As before, we set $F = \vv^n  f$.  Notice that $\vv^n \partial_{v_i} f = \partial_{v_i} F - n\vv^{n-2} v_i f$.  Hence,
\begin{equation}\label{e:c31}
\begin{split}
	\int  Q(g,  f) \partial f \vv^{2n} \dd v
		&= \int [\vv^n Q(g,  f) -
					Q(g, \vv^n f)] \vv^n \partial f  \dd v\\
		&\qquad  
			- \int Q(g, F) f n v_i \vv^{n-2} \dd v
			+ \int Q(g, F) \partial F \dd v. 
\end{split}
\end{equation}

\medskip

{\bf The first term in~\eqref{e:c31}.}
We use the commutator estimate of \Cref{p:commutator}.  Indeed, we find, for $\e>0$ satisfying $(2s-1)_+ + \e < 1$, 
\begin{equation}\label{e:Q_civ1}
\begin{split}
	\int &[\vv^n Q(g, f) - Q(g, \vv^n  f)] \vv^n \partial f  \dd v\\
		&\lesssim \|g\|_{L^{2,n + 3/2 + (2s-1)_+ + \e}} \|f\|_{H^{(2s-1+\e)_+, n + (2s-1)_+}} \|\partial f\|_{L^{2,n}}.
\end{split}
\end{equation}

\medskip

{\bf The second term in~\eqref{e:c31}.} We may now appeal to \Cref{t:amuxy-sobolev-v} (with $\theta = -1$), which yields
\begin{equation}\label{e:Q_civ2}
	\Big| \int Q(g,F) (f v_i \vv^{n-2}) \dd v\Big|
		\lesssim \left( \|g\|_{L^{1,(\gamma+2s)_+}} + \|g\|_{L^2_v}\right) \|F\|_{H^{(2s-1+\eps)_+, (\gamma + 2s)_+}} \|f v_i \vv^{n-2}\|_{H^1}.
\end{equation}
This yields the desired estimate after using~\eqref{e:L1_L2}.

\medskip

{\bf The third term in~\eqref{e:c31}.} This term requires the most work.   Before beginning, we describe the obstruction.  In this term, we are able to distribute derivatives over each $F$ term; however, we want an estimate in terms of no more than $1$ derivatives of $f$; however, $Q(g,F)$ can be thought of as $2s$-derivatives applied to $F$.  Hence, we have two $F$ terms and $1 + 2s$ derivatives to distribute over them.  Without appealing to a special symmetry, one of the $F$ terms must accept at least $2s$-derivatives, which makes our proof impossible.  Thus, we search for a symmetry to aid us.

We decompose $Q$ into $\Qs$ and $\Qns$.  The nonsingular term is easy to estimate.  Indeed, recalling ~\Cref{l:Qns}, we have
\begin{equation}\label{e:Q_civ31}
\begin{split}
	\int &\Qns(g,F) \partial F \dd v
		= \int (S* g) F \partial F \dd v
		= - (1/2)\int S*(\partial g) F^2 \dd v \\
		&\lesssim \|S*(\partial g)\|_{L^\infty} \|F\|_{L^2}^2 
		\lesssim \|\partial g\|_{L^\infty}^\frac{|\gamma|}{3}\|\partial g\|_{L^1}^\frac{3 + \gamma}{3} \|F\|_{L^2}^2
		\lesssim (\| g\|_{W^{1,\infty}} + \|g\|_{H^{1,3/2 + \e}}) \|f\|_{L^{2,n}}^2.
\end{split}
\end{equation}

The singular term requires more care.  First, we symmetrize the equation to obtain
\[\begin{split}
	2\int &\Qs(g, F) \partial F \dd v
		= \int K_g (F' - F)(\partial F - (\partial F)') \dd v' \dd v
			+ \int (K_g - K_g') (F' - F) \partial F \dd v' \dd v
		=:I_1 +I_2,
\end{split}\]
where we use $K_g$ to denote $K_g(v,v')$ and $K_g'$ to denote $K_g(v',v)$.  Similarly, let $\partial'$ denote $\partial_{v_i'}$, which is useful in the sequel.  With this notation, $(\partial F)' = \partial' F'$.

\medskip

{\bf Estimating $I_1$.}  Notice that
\[\begin{split}
	2I_1
		&=  -\int K_g (\partial + \partial') (F' - F)^2\dd v' \dd v
		= \int (\partial + \partial') K_g (F' - F)^2 \dd v' \dd v
		= 2\int K_{\partial g} (F' - F)^2 \dd v' \dd v.
\end{split}\]
Denote the diagonal strip $S_D = \{(v,v') : |v-v'| < 1\}$.  Then
\[
	\int K_{\partial g} (F' - F)^2 \dd v' \dd v
		= \int_{S_D} K_{\partial g} (F' - F)^2 \dd v' \dd v
			+ \int_{S_D^c} K_{\partial g} (F' - F)^2 \dd v' \dd v
		=: I_{11} + I_{12}.
\]

For $I_{11}$, we decompose the integral into a sum over integrals on compact sets and apply a cut-off.  Indeed, let $p = (\gamma + 2s)_+$, and let $\chi$ be any smooth function that is one in $B_{10}$ and zero outside of $B_{20}$.  Then,
\[\begin{split}
	|I_{11}|
		&\leq \sum_{z\in \Z^3} \int_{B_{10}(z,z)} K_{|\partial g|} (F'-F)^2 \dd v' \dd v 
		\lesssim \sum_{z\in \Z^3} \int K_{|\partial g|} (\chi' F'- \chi F)^2 \dd v' \dd v.
\end{split}\]

We can now appeal to \cite[Theorem 4.1]{imbert2016weak} to find
\begin{equation}\label{e:Q_civ32}
\begin{split}
 	|I_{11}|
		&\lesssim \sum_{z\in \Z^3} 
			\langle z \rangle^{\gamma + 2s} \|\partial g\|_{L^{2, \eta}} \|\chi F\|_{H^s}^2
		\lesssim \sum_{z\in \Z^3} \frac{\|\partial g\|_{L^{2, \eta}} \| F\|_{H^{s,p/2}(B_{20}(z))}^2}{\langle z \rangle^{p- (\gamma + 2s)}}
		\lesssim \|\partial g\|_{L^{2, \eta}} \|F\|_{H^{s,p/2}}^2.
\end{split}
\end{equation}

The bound of $I_{12}$ is simpler.  We first use Young's inequality:
\begin{equation*}
	|I_{12}|
		\lesssim \int_{S_D^c} K_{|\partial g|} (F')^2 \dd v' \dd v  + \int_{S_D^c} K_{|\partial g|} F^2 \dd v' \dd v.
\end{equation*}
Both terms above are bounded in essentially the same way; hence, we only include the proof of the bound for the first term as it is slightly more involved.  Using \Cref{l:kernel_bounds}, we find
\[\begin{split}
	\int_{S_D^c} &K_{|\partial g|} (F')^2 \dd v' \dd v 
		= \int_{\R^3} F(v')^2 \left(\int_{B_1(v')^c} K_{|\partial g|}(v,v') \dd v\right) \dd v'\\
		&\lesssim \int_{\R^3} F(v')^2 \vvp^{\gamma + 2s} \|\partial g\|_{L^{2, 3/2 + (\gamma + 2s)_+ + \e}} \dd v' 
		\leq \|F\|_{L^{2, \gamma/2+s}}^2 \|\partial g\|_{L^{2, \eta}}.
\end{split}\]
Thus,
\begin{equation}\label{e:Q_civ33}
	|I_{12}|
		\lesssim \|F\|_{L^{2, \gamma/2+s}}^2 \|\partial g\|_{L^{2, \eta}}.
\end{equation}
This concludes the proof of the bound of $I_1$.

\medskip

{\bf Estimating $I_2$.} 
Recall the definition of $S_D$ from above.  Then
\[
	I_2
		= \int_{S_D} (K_g - K_g')(F' - F) \partial F \dd v' \dd v
			+ \int_{S_D^c}(K_g - K_g')(F' - F) \partial F \dd v' \dd v
		=: I_{21} + I_{22}.
\]
Since the estimate of $I_{22}$ is more straightforward, we begin with that.  Indeed, we split the integral in two:
\begin{equation}\label{e:c36}
	I_{22}
		= \int_{S_D^c} (K_g - K_g') F' \partial F \dd v' \dd v 
			- \int_{S_D^c} (K_g - K_g') F \partial F \dd v' \dd v.
\end{equation}
Using \Cref{l:kernel_bounds}, it is clear that
\[\begin{split}
	\Big|\int_{S_D^c} (K_g - K_g') F \partial F \dd v' \dd v\Big|
		&\lesssim (\|g\|_{L^{2,3/2 + \e}} + \|g\|_{L^\infty}) \int \vv^\gamma |F \partial F| \dd v\\
		&\leq (\|g\|_{L^{2,3/2 + \e}} + \|g\|_{L^\infty})\|F\|_{L^2} \|F\|_{H^1}.
\end{split}\]

The first term in~\eqref{e:c36} is estimated similarly, though with a bit more subtlety.  Indeed, from \Cref{l:kernel_bounds}, we see that
\[\begin{split}
	\int_{S_D^c} &(K_g - K_g') F' \partial F \dd v' \dd v
		\lesssim \int_{S_D^c}\frac{\|g\|_{L^{\infty, 3/2 + \eta}}\left(\vv^{\gamma + 2s + 1} + \vvp^{\gamma + 2s + 1}\right)}{|v-v'|^{3 + 2s}} |F'| |\partial F| \dd v' \dd v.
\end{split}\]
Since $|v-v'| \geq 1$, we have that $\vv^{\gamma + 2s + 1} |v-v'|^{-1 - \gamma-2s} \lesssim \vvp^{(\gamma+2s+1)_+}$.  Thus,
\[\begin{split}
	\int_{S_D^c} &(K_g - K_g') F' \partial F \dd v' \dd v
		\lesssim \int_{S_D^c}\frac{\|g\|_{L^{\infty, 3/2 + \eta}}\vvp^{(\gamma + 2s + 1)_+}}{|v-v'|^{2 -\gamma}} |F'| |\partial F| \dd v' \dd v\\
		&= \|g\|_{L^{\infty, 3/2 + \eta}} \int_{\R^3} \vvp^{(\gamma+2s+1)_+} |F'| \left(\int_{B_1(v')^c} \frac{|\partial F|}{|v-v'|^{2-\gamma}} \dd v\right) \dd v'\\
		&\lesssim \|g\|_{L^{\infty, 3/2 + \eta}} \|F\|_{L^{2,(\gamma+2s+1)_+ +  3/2 + \e}} \|\partial F\|_{L^2}
		\lesssim \|g\|_{L^{\infty, 3/2 + \eta}} \|f\|_{L^{2,n + 1 + \eta}} \|f\|_{H^{1,n}}.
\end{split}\]
In the second-to-last inequality we used Cauchy-Schwarz twice and that $2 - \gamma > 3/2$.  The last equality is just by recalling the definition of $F$.  Thus, we have established
\begin{equation}\label{e:Q_civ34}
	|I_{22}|
		\lesssim (\|g\|_{L^{2,3/2+\e}} + \|g\|_{L^{\infty, 3/2 + \eta}}) \|f\|_{L^{2,n + 1 + \eta}} \|f\|_{H^{1,n}}.
\end{equation}

Now we consider $I_{21}$.  First, let $p = (\gamma + 2s + 1)_+$.  Then we re-write $I_{21}$:
\[\begin{split}
	I_{21}
		&= \int_{S_D} \left[\left(\frac{K_g - K_g'}{\vv^p} ((\vv^p F)' - (\vv^p F)) \partial_{v_i}F\right)
			+ \left(\frac{K_g - K_g'}{\vv^p} (\vvp^p - \vv^p ) F' \partial F\right) \right]\dd v' \dd v \dd x\\
		&=: I_{211} + I_{212}.
\end{split}\]

We first obtain a useful bound on the kernel.  Let $\alpha \in (s,2s)$ and let $\theta = 2s-\alpha$.  Notice that $\theta \in (0,s)$.  Then, using  \Cref{l:kernel_bounds} and recalling that $(v,v') \in S_D$ implies $\vv \approx \vvp$, we find
\begin{equation*}
	|K_g - K_g'|
		\lesssim \|\langle \cdot \rangle^{2 + p + \e}g\|_{C_v^{\alpha}} \vv^{\gamma + 2s + 1} |v-v'|^{-3 - \theta}.
\end{equation*}

We estimate $I_{211}$ first.  Let $\e' \in (0, s-\theta)$. Using the bound on $K_g$ above and Cauchy-Schwarz, we see that
\[\begin{split}
	|I_{211}|
		&\lesssim \|\langle \cdot \rangle^{2 + p + \e}g\|_{C_v^{\alpha}} \int_{S_D}  \left(\frac{|(\vv^p F)' - (\vv^pF)|}{|v-v'|^{3/2+\theta+ \e'}}\right)\left(\vvp^{\gamma + 2s + 1 - p} \frac{|\partial F|}{|v-v'|^{3/2-\e'}}\right) \dd v' \dd v\\
		&\lesssim \|\langle \cdot \rangle^{2 + p + \e}g\|_{C_v^{\alpha}}\left(\int_{S_D}  \frac{|(\vv^p F)' - (\vv^pF)|^2}{|v-v'|^{3+2(\theta + \e')}} \dd v \dd v' \right)^{1/2}
		 \left(\int_{S_D}  \frac{|\partial F|^2}{|v-v'|^{3 - 2\e'}}\dd v \dd v'\right)^{1/2}.
\end{split}\]
We used here that $\gamma + 2s + 1 - p \leq 0$.

The first parenthetical term is, up to extending the domain of integration, $\int (\vv^p F) (-\Delta)^{\theta+\e'} (\vv^p F) \dd v$.  Thus, it is bounded by $\|F\|_{H^{\theta+\e',p}}$, which is, in turn, bounded by $\|F\|_{H^{s,p}}$.  On the other hand, the second parenthetical term is bounded by $\|\partial F\|_{L^2}$, by integrating first in $v'$ and then in $v$. We conclude that
\begin{equation}\label{e:Q_civ35}
	|I_{211}|
		\lesssim \|\langle \cdot \rangle^{2 + p + \e}g\|_{C_v^{\alpha}}
			\|f\|_{H^{s, n + (\gamma + 2s + 1)_+}}
			\|f\|_{H^{1,n}}.
\end{equation}

We now turn to $I_{212}$.  This follows by an easy application of Taylor's theorem (again, using that $\vv \approx \vvp$ on $S_D$).  This yields
\begin{equation}\label{e:Q_civ36}
	|I_{212}|
		\lesssim \|\langle \cdot \rangle^{2 + p + \e}g\|_{C^{\alpha}}
			\|f\|_{L^{2, n + (\gamma + 2s)_+}}
			\|f\|_{H^{1,n}}
\end{equation}
but omit the details as it is simpler than the argument for $I_{211}$.

The combination of \cref{e:Q_civ1,e:Q_civ2,e:Q_civ31,e:Q_civ32,e:Q_civ33,e:Q_civ34,e:Q_civ35,e:Q_civ36} concludes the proof.
\end{proof}

\subsection{Weighted $L^\infty$ bounds on $Q$: the proof of \Cref{p:Q_estimates}.(v)}

\begin{proof}[Proof of \Cref{p:Q_estimates}.(v)]
	First, we estimate $Q_{\rm ns}(g,f)$.  Using Lemma \ref{l:Qns}, we find immediately that
	\[
		\|Q_{\rm ns}(g,f)\|_{L^{\infty,m}}
			\leq \|g\|_{L^{\infty,3}} \|f\|_{L^{\infty,m}}.
	\]
	Hence, the bulk of the work is in estimating $Q_{\rm s}(g,f)$.  
	
	Let $A_k(v) = B_{|v| 2^{k}}(v)\setminus B_{|v| 2^{k-1}}(v)$.  Then we write
	\[
		\Qs(g,f)
			= \sum_{k\in \Z} \int_{A_k(v)} (f(v') - f(v)) K_g(v,v') \dd v'.
	\]
	Bounding the terms $k \leq -1$, $k\geq 0$ when $|v| \leq 10$, and $k\geq 0$ over the domain $A_k(v) \cap B_{\vv/2}^c$ when $|v| \geq 10$ is exactly as in \Cref{p:Q_estimates}.(i).  Hence, we omit them and consider the remaining case.
	
	Fix any $k \geq 0$. Suppose that $|v| \geq 10$ and consider the integral over the set $A_k(v) \cap B_{\vv/2}$.  Since $|v| \geq 10$, then $\vv \leq 10 |v|/11$.  It follows that $A_k(v) \cap B_{\vv/2} = \emptyset$ unless $k = 0$ or $k=1$.  We prove the case $k=0$ but the case $k=1$ follows similarly.

	We claim that, for all $v' \in B_{\vv/2}(0)$,
	\begin{equation}\label{e:c2}
		|K_g(v,v')|
			\lesssim \|g\|_{L^{\infty,m}}\vv^{-m + \gamma}.
	\end{equation}
	We postpone the proof of \eqref{e:c2} and show how to conclude using it.  Note that
	\begin{equation}\label{e:Q_cv4}
	\begin{split}
		\left|\int_{B_{R/2} \cap A_0(v)} ( f(v') - f(v)) K_g(v,v') \dd v'\right|
			&\lesssim \int_{B_{R/2} \cap A_0(v)} ( f(v')+  f(v)) \|g\|_{L^{\infty,m}}\vv^{-m + \gamma}\dd v\\
			 & \lesssim \|g\|_{L^{\infty,m}}\vv^{-m + \gamma}.
	\end{split}
	\end{equation}
	Thus, the proof is completed 
	up to establishing~\eqref{e:c2}.

	We now prove~\eqref{e:c2}.  
	Using~\eqref{e:andrei1}, if $w \in (v' - v)^\perp$, then $|v+w|^2 \approx |v|^2 + |w|^2$. 
	Thus,
	\[\begin{split}
		|K_g(v,v')|
			&= \frac{1}{|v-v'|^{3+2s}} \Big|\int_{(v-v')^\perp} g(v+w) |w|^{\gamma + 2s + 1} \dd w\Big|\\
			&\lesssim \frac{1}{|v-v'|^{3+2s}} \|g\|_{L^{\infty,m}} \int_{(v-v')^\perp} (|v| + |w|)^{-m}|w|^{\gamma + 2s + 1} \dd w
			\approx \frac{|v|^{-m + \gamma + 2s + 3}}{|v-v'|^{3+2s}} \|g\|_{L^{\infty,m}}.
	\end{split}\]
Next, using that $|v'| \leq \vv/2 \leq 11|v|/20$, we have $|v'-v|\gtrsim |v|$.  Using this in the above estimate of $K_g(v,v')$ concludes the proof of~\eqref{e:c2} and, thus, the proposition.	
\end{proof}

\appendix

\section{Technical Lemma}\label{s:appendix}

The following inequality is used crucially above in Step Three of \Cref{p:Q_estimates}.(i).  When $v_0 = 0$, it follows from a change of variables (see \cite[Lemma A.10]{imbert2016weak}); however, when $v_0 \neq 0$, the geometry becomes non-trivial and so a proof is required.

\begin{lemma}\label{l:henderson-problem}
For any $\rho>0$ and $v_0\in \R^3$ such that $\rho \geq 2|v_0|$, and any $H:\R^3\to [0,\infty)$ such that the right-hand side is finite, we have
\begin{equation}\label{e:henderson-problem}
	\int_{\partial B_\rho(0)} \int_{\{w\in \R^3 : w \cdot (z-v_0) = 0\}} H(z+w) \dd w \dd z 
		\lesssim \rho^2  \int_{B_{\rho/2}^c(0)} \frac {H(w)}{|w|} \dd w.
\end{equation}
\end{lemma}
\begin{proof}
First, we notice that we may assume that $\rho=1$ without loss of generality by scaling.  
Next, we notice that if $w \cdot (z-v_0) = 0$ and $|z| = 1$, then $z + w \in B_{\sqrt 3 / 2}^c$.  Indeed,
\[\begin{split}
	|z+w|^2 &= |z|^2 + 2(z-v_0)\cdot w + 2 v_0 \cdot w + |w|^2
		\geq |z|^2 - |v_0|^2
		\geq 1 - (1/2)^2 = 3/4.
\end{split}\]

By density and linearity, we may assume that $H = \1_{B_r(\bar v)}$ where $r\in (0, 1/4)$, and up to rotation, we can assume that $\bar v + v_0$ lies on the positive $x$-axis.  Also, we may assume that $|\bar v| \geq 3/5$ because, if not, then the left hand side of~\eqref{e:henderson-problem} is, by our computation above, zero.  

The right hand side of~\eqref{e:henderson-problem} is easy to compute.  Since $|\bar v| \geq 1/2$, we find
\begin{equation}\label{e:henderson1}
	\int_{B_{1/2}^c} \frac{|H(w)|}{|w|} \dd w
		= \int_{B_{1/2}^c \cap B_r(\bar v)} \frac{1}{|w|} \dd w
		\gtrsim \int_{B_{1/2}^c \cap B_r(\bar v)} \frac{1}{|\bar v|} \dd w
		\geq \frac{2 \pi r^3}{3|\bar v|}.
\end{equation}
The last inequality holds because $B_{1/2}^c \cap B_r(\bar v)$ contains half of a ball of radius $r$.

We now consider the left hand side of~\eqref{e:henderson-problem}.  Let 
\[
	\cA = \{z \in \partial B_1 : \text{ there exists } w\in B_r(\bar v - z) \text{ such that } w\cdot (z - v_0) = 0\}.
\]
Fix any $z\in \cA$.  Then $\{w \in B_r(\bar v - z) : w\cdot(z-v_0) = 0\}$ is the intersection of a plane with a ball of radius $r$ and so can have measure at most $\pi r^2$.  If $z\notin \cA$, then $\{w \in B_r(\bar v - z) : w\cdot(z-v_0) = 0\} = \emptyset$. Thus,
\[
	\int_{\partial B_1(0)} \int_{\{w\in \R^3 : w \cdot (z-v_0) = 0\}} |H(z+w)| \dd w \dd z
		\leq \int_{\cA}  \pi r^2 \dd z
		= \pi r^2 |\cA|.
\]
In view of this and~\eqref{e:henderson1}, we are finished if we establish that
\begin{equation}\label{e:henderson2}
	|\cA|
		\lesssim \frac{r}{|\bar v|}.
\end{equation}

For any $z\in \partial B_1$, $z\in \cA$ if and only if the distance between the plane $P_z = \{ u \in \R^3 : (u - z) \cdot (z-v_0) = 0\}$ and $\bar v$ is less than $r$.  Hence, using elementary linear algebra, if $z\in \cA$,
\[
	r
		\geq \left| \frac{z-v_0}{|z-v_0|} \cdot (\bar v - z)\right|
		\geq \frac{2}{3} |z \cdot \bar v - z\cdot z - v_0 \cdot \bar v + v_0 \cdot z|
		= \frac{2}{3} |z \cdot(\bar v + v_0) - (1 + v_0 \cdot \bar v)|.
\]
Importantly, we have that $|\bar v| \geq 3/5$ and $|v_0|\leq 1/2$ so that $|\bar v + v_0| \gtrsim |\bar v|$.  Then the above implies that
\[
	\Big|\cos(\theta) 
		- \frac{1 + v_0\cdot \bar v}{|\bar v + v_0|}\Big|
		\leq \frac{Cr}{|\bar v|},
\]
where $\theta$ is the angle between $z$ and $\bar v + v_0$ and $C$ is a universal constant.

Let $\theta\in [0,\pi]$ be the angle between $z$ and $\bar v + v_0$.  Let $\theta_-, \theta_+ \in [0,\pi]$ be defined by
\[
	\cos(\theta_-)
		= \max\left\{\frac{1 + v_0\cdot \bar v}{|\bar v + v_0|}
			- \frac{Cr}{|\bar v|}, -1 \right\}
	\quad\text{ and }\quad
	\cos(\theta_+)
		= \min\Big\{\frac{1 + v_0\cdot \bar v}{|\bar v + v_0|}
			+ \frac{Cr}{|\bar v|}, 1\Big\}.
\]
It follows that $\theta \in [\theta_+, \theta_-]$. 

The $z$ such that $\theta \in [\theta_+,\theta_-]$ make up a set that is symmetric about the $x$-axis set with angles between $[\theta_+,\theta_-]$ (recall that we assumed that $\bar v + v_0$ lie on the positive $x$-axis).  By elementary calculus, the area of this set is
\[
	2\pi \int_{\theta_+}^{\theta_-} \sin(\theta) \dd \theta
		= 2\pi \left[ \cos(\theta_+) - \cos(\theta_-)\right]
		\leq \frac{4\pi C r}{|\bar v|}.
\]
We conclude that $|\cA| \lesssim r/|\bar v|$; that is,~\eqref{e:henderson2} holds, and, hence, the proof is complete.
\end{proof}

\bibliographystyle{abbrv}
\bibliography{boltz_existence.bib}

\end{document}